\documentclass[11pt,reqno]{amsart}
\usepackage{amsmath}
\usepackage{hyperref} 
\usepackage{amsfonts}
\usepackage{amssymb}
\usepackage{amsthm}
\usepackage[arrow, matrix, curve]{xy} 
\usepackage{hyperref}
\usepackage{svn-multi}
\usepackage{amssymb,amscd,url,tikz,stmaryrd,mathtools,fixltx2e}
\usepackage{color}
\usepackage{tikz, pgfplots}
\usepackage{tikz-cd} 
\usepackage{comment}

\usepackage{paralist,csquotes,stackrel}
 
\usepackage[mathscr]{eucal}

\numberwithin{equation}{section}

\newtheorem{proposition}{\textbf{Proposition}}
\newtheorem{prop}[proposition]{\textbf{Proposition}}
\newtheorem{lemma}[proposition]{\textbf{Lemma}}
\newtheorem{lem}[proposition]{\textbf{Lemma}}
\newtheorem{corollary}[proposition]{\textbf{Corollary}}
\newtheorem{cor}[proposition]{\textbf{Corollary}}
\newtheorem{theorem}[proposition]{\textbf{Theorem}}
\newtheorem{thm}[proposition]{\textbf{Theorem}}

\theoremstyle{definition}
\newtheorem{definition}[proposition]{\textbf{Definition}}

\newtheorem{remark}[proposition]{\textbf{Remark}}
\newtheorem{rem}[proposition]{\textbf{Remark}}

\newcommand{\Lie}[1]{\operatorname{\textsl{#1}}}

\newcommand{\lie}[1]{\operatorname{\mathfrak{#1}}}

\newcommand{\sln}{\lie{sl}}

\newcommand{\SO}{\Lie{SO}}

\newcommand{\SL}{{\rm SL}}

\newcommand{\GL}{{\rm GL}}

\newcommand{\SU}{{\rm SU}}

\newcommand\C{{\mathbb C}}

\newcommand{\R}{{\mathbb R}}

\newcommand{\End}{{\rm End}}

\newcommand{\Hom}{{\rm Hom}}

\newcommand{\Z}{\mathbb{Z}}

\numberwithin{proposition}{section}

\begin{document}

\title[Energy of sections  of Deligne--Hitchin twistor space]{Energy of sections  of the Deligne--Hitchin twistor space}

\author[F. Beck]{Florian Beck}
\email[(Beck)]{florian.beck@uni-hamburg.de}
\author[S. Heller]{Sebastian Heller}
\email[(Heller)]{seb.heller@gmail.com}
\author[M. R\"oser]{Markus R\"oser}
\email[(R\"oser)]{markus.roeser@uni-hamburg.de}
\address[Beck, R\"oser]{Fachbereich Mathematik,
Universit\"at Hamburg, 20146 Hamburg, Germany} 
\address[Heller]{Insitut f\"ur Differentialgeometrie, Leibniz Universit\"at Hannover, Welfengarten 1, 30167 Hannover}


\date{\today}

\begin{abstract}
We study a natural functional on the space of holomorphic sections of the Deligne--Hitchin moduli space of a
compact Riemann surface, generalizing the energy of equivariant harmonic maps corresponding to twistor lines.
We show that the energy is the residue of the pull-back along the section of a natural meromorphic connection on the hyperholomorphic line bundle recently constructed by Hitchin. As a byproduct, we show the existence of a hyper-K\"ahler potentials for 
new components of real holomorphic sections of twistor spaces of hyper-K\"ahler manifolds with rotating $S^1$-action.
Additionally, we prove that for a certain class of real holomorphic sections of the Deligne--Hitchin moduli space, the energy functional is basically given by the Willmore energy of corresponding equivariant conformal map to the 3-sphere. As an application
we use the functional to distinguish new components of real holomorphic sections of the Deligne--Hitchin moduli space from the space of twistor lines.
\end{abstract}
\maketitle

\section*{Introduction}
The Deligne--Hitchin moduli space $\mathcal M_{DH}(\Sigma, G_\C)$ \cite{Si-Hodge} for a compact Lie group $G$ with complexification $G_\C$, is a complex analytic reincarnation of the twistor space of the hyper-K\"ahler moduli space $\mathcal M_{SD}(\Sigma,G)$ of solutions of Hitchin's self-duality equations on a principal $G$-bundle over a compact Riemann surface $\Sigma$ \cite{Hitchin_1987}. It is defined by gluing the Hodge moduli spaces of $\lambda$-connections on $\Sigma$ and $\overline{\Sigma}$ via the monodromy representation, and it admits a holomorphic fibration over the projective line. The fibers are the moduli spaces of $G_\C$-Higgs bundles, flat $G_\C$-connections or $G_\C$-Higgs bundles over $\overline \Sigma$. Holomorphic sections of the Deligne--Hitchin moduli space are interesting for various reasons: $\mathcal M_{DH}(\Sigma, G_\C)$ carries an anti-holomorphic involution $\tau$ covering the antipodal map $\lambda\mapsto -\bar\lambda^{-1}$. By the twistor construction for hyper-K\"ahler manifolds \cite{HKLR}, the hyper-K\"ahler moduli space $\mathcal M_{SD}(\Sigma,G)$ can be identified with a certain component of the space of $\tau$-real holomorphic sections of the fibration $\mathcal{M}_{DH}(\Sigma,G_\C)\to\C P^1$. These sections are called twistor lines. On the other hand, a solution of the self-duality equations corresponds to an equivariant harmonic map from the universal cover $\tilde \Sigma$ into the symmetric space $G_\C/G$, which can be reconstructed from the associated twistor line by loop group factorization methods \cite{Si-construct,BHR}.
Apart from the twistor lines, holomorphic sections satisfying other types of reality conditions arise from equivariant harmonic maps of $\tilde\Sigma$ into different (pseudo-)Riemannian
symmetric spaces related to the the group $G_\C$ and its real forms. 

This paper is motivated by the work of some of the authors about the question of Simpson, whether all $\tau$-real holomorphic sections are in fact twistor lines \cite{Si-Hodge}. The answer turns out to be \emph{no} \cite{HH2,BHR}, and leads to the problem of how to differentiate between the components of the space of $\tau$-real holomorphic sections. 

The most fundamental quantity associated to a harmonic map is its energy and the starting point of this paper is the
simple observation that the
energy of a harmonic map (defined on a compact Riemann surface) can be computed via its associated holomorphic section
of the Deligne--Hitchin moduli space (see Theorem \ref{harmonicenergy}). This computation leads us to a well-defined energy functional on the space of holomorphic sections (see
Proposition \ref{defene}). The detailed investigation of this functional is the main objective of our work. 
It should be mentioned that the functional is defined in terms of the complex analytic structure of the Deligne--Hitchin moduli space and its definition does not involve the hyper-K\"ahler metric on $\mathcal M_{SD}(\Sigma,G)$, i.e. the twistor lines.

We will mostly be concerned with the case $G = \SU(2)$, so that $G_\C = \SL(2,\C)$. 
Twistor lines then correspond to equivariant harmonic maps of $\tilde\Sigma$ into the hyperbolic space $H^3 = \SL(2,\C)/\SU(2)$. This is the space of positive definite hermitian matrices, hence these harmonic maps are called harmonic metrics. One can also study equivariant harmonic maps from $\tilde\Sigma$ into the $3$-sphere $S^3 = \SU(2)$ (the compact dual of $\SL(2,\C)/\SU(2)$), into the anti de Sitter space $\mathrm{AdS}^3 = \SL(2,\R)$ or into the de Sitter space $\mathrm{dS}^3 = \SL(2,\C)/\SL(2,\R)$ via holomorphic sections of $\mathcal M_{DH}(\Sigma,\SL(2,\C))\to\mathbb CP^1$. To a conformal equivariant harmonic map one may, under certain circumstances, associate another harmonic map with in general different target space.
This process, which we call \emph{twisting}, played a central role in the construction of counterexamples to Simpson's question in \cite{BHR} and we will give a more systematic treatment in this article. 
We study how the energy functional interacts with the two real structures on $\mathcal{M}_{DH}(\Sigma, G_\C)$. We will see that it takes real values on real holomorphic sections and is normalised in such a way that it takes non-positive values on twistor lines, while it is non-negative on holomorphic sections corresponding to equivariant harmonic maps into $G$. In the rank $2$ case we then examine its behavior under twisting (see Proposition \ref{twisttheenergy}). The explicit relation between the energy of a section and its twist allows us to give an alternative proof that the $\tau$-real sections constructed in \cite{BHR} are not twistor lines by checking that these have positive energy.  

Although the definition of the functional is motivated by the theory of harmonic maps and does not involve the hyper-K\"ahler structure of $\mathcal M_{SD}(\Sigma,G)$, it can be given a natural interpretation in terms of the hyper-K\"ahler geometry of the moduli space  $\mathcal M_{SD}(\Sigma,G)$. The natural isometric circle action on $\mathcal M_{SD}(\Sigma,G)$ plays a central role. It preserves one of the complex structures of and rotates the other two complex structures. We show that an analogous functional exists on the space of holomorphic sections of the twistor space of any hyper-K\"ahler manifold with an isometric circle action of this type. Building on work by Haydys \cite{Haydys}, Hitchin \cite{Hitchin-hkqk} has shown that on the twistor space of such a hyper-K\"ahler manifold, one has a natural holomorphic line bundle with meromorphic connection. The pull-back of the meromorphic connection along a holomorphic section of $\mathcal M_{DH}(\Sigma,G_\C)\to\mathbb CP^1$ has simple poles at $\lambda=0,\infty$ only, and it turns out that the residue at $\lambda=0$ coincides with the energy (Corollary \ref{cor:ResEnergy}). As a byproduct, we show that the residue evaluation along sections is always a complexification
of the moment map of the $S^1$-action (Theorem \ref{thm:res}).  Moreover, it automatically serves as a K\"ahler potential on all hyper-K\"ahler components of real holomorphic sections of the twistor space. Recently it has been shown \cite{Hel} that there indeed exist such hyper-K\"ahler components of the space real holomorphic sections of $\mathcal M_{DH}(\Sigma,\SL(2,\C))$. The energy functional thus gives a K\"ahler potential on these new components and we hope to extract from it more information about the geometry of these components in the future.

The third main objective of the paper is the geometric interpretation of the energy for a class of
 $\tau$-real holomorphic sections which are not twistor lines \cite{HH2}.
 Recall that in the case of $G=\SU(2)$, twistor lines correspond to equivariant harmonic maps to hyperbolic 3-space $H^3.$
  A holomorphic section of the type constructed in \cite{HH2} is instead obtained from a M\"obius equivariant Willmore surface $\tilde\Sigma\to S^3$. By decomposing the 3-sphere $S^3 = H^3\cup S^2\cup H^3$
 into two hyperbolic balls separated by the boundary 2-sphere at infinity,
  one can show that such a holomorphic section defines a solution of the self-duality equations on an open dense subset of the Riemann surface $\Sigma$. The solution blows up in a well-behaved way near certain curves on the surface. The corresponding equivariant harmonic map into hyperbolic $3$-space intersects $S^2$, the boundary at infinity, along these curves and continues as a harmonic map on the other side.  We prove that the energy of such a section is directly related to the Willmore energy of the surface, a conformally invariant measure of the roundness of an immersed surface. 
This relation allows us to prove our last main result: the sections constructed in \cite{HH2} have positive energy. This gives
a complex analytic way to distinguish the component of twistor lines from this newly discovered component of real holomorphic sections.

The structure of the paper is as follows. In Section \ref{se1} we set up some notation and recall basic notions associated with holomorphic sections of the Deligne--Hitchin moduli space over a compact Riemann surface. In Section \ref{Sec:defene} we define the energy functional and
prove its basic properties. Section \ref{hyperenergy} then contains
 the natural interpretation of the energy functional in terms of the residue of the meromorphic connection on the hyperholomorphic line bundle over $\mathcal M_{DH}(\Sigma,G_\C)$. 
 In Section \ref{Sec:Willener}, we relate the energy of the real holomorphic sections constructed in \cite{HH2} to the Willmore energy of the related M\"obius equivariant Willmore surfaces.
In the final Section \ref{Sec:enes}, we show that the energy functional can be used to distinguish different components of the space of real holomorphic sections. In particular, we prove that the new sections of \cite{HH2} have positive energy.

\section{The Deligne--Hitchin moduli space}\label{se1}
\subsection{$\lambda$-connections and the Deligne--Hitchin moduli space}\label{ss:dh}
Let $(M^{4k};g,I_1,I_2,I_3)$ be a hyper-K\"ahler manifold. Recall that this means that $g$ is a Riemannian metric and $I_1,I_2,I_3$ are orthogonal complex structures satisfying the quaternionic relations $I_1I_2 = I_3 = -I_2I_1$ such that the two-forms $\omega_j = g(I_j-,-)$, $j=1,2,3$ are closed. It can be shown that $\omega_\C = \omega_2+i\omega_3$ is a holomorphic symplectic form with respect to the complex structure $I_1$.  

Associated with a hyper-K\"ahler manifold we have the twistor space $Z = Z(M)$ which is a complex manifold of complex dimension $2k+1$ on which the hyper-K\"ahler structure is encoded in the following complex-geometric data \cite{HKLR}: 
\begin{itemize}
\item a holomorphic projection $\pi_Z:Z\to\mathbb CP^1$,
\item a holomorphic section $\omega\in H^0(\Lambda^2T_F^*(2))$ inducing a holomorphic symplectic form on each fibre $\pi^{-1}(\lambda)$ (here $T_F = \ker d\pi_Z$ is the tangent bundle along the fibers),
\item an anti-holomorphic involution $\tau_Z: Z\to Z$ covering the antipodal map $\mathbb CP^1\to\mathbb CP^1$ and such that $\overline{\tau_Z^*\omega} = \omega$,
\item a family (parametrized by $M$) of $\tau_Z$-real holomorphic sections with normal bundle isomorphic to $\mathcal O(1)^{2k}$ , the \emph{twistor lines}.  
\end{itemize}
 
We now briefly recall the construction of the Deligne--Hitchin moduli space, which may be interpreted as the twistor space of the hyper-K\"ahler moduli space of solutions to the self-duality equations on a Riemann surface $\Sigma$. For details we refer to \cite{Si-Hodge}, see also \cite{BHR} for a more differential geometric account. The discussion of this subsection works for complex reductive Lie groups $G$ as structure groups. 
Since we fully work out our concepts, e.g. twisting (Section \ref{Subsec:Twisting}), for $\SL(2,\C)$, we choose $G=\SL(n,\C)$ in this subsection for concreteness. 

Let $\Sigma$ be a compact Riemann surface and denote by $E\to\Sigma$ the trivial smooth rank $n$ vector bundle. We endow $E$ with an $\SL(n,\C)$-structure, i.e. a trivialisation $\det E \cong\mathcal O_\Sigma$, which in the case of rank $2$ is a complex symplectic form.  
We denote by $\sln(E)$ the subbundle of $\mathrm{End}(E)$ given by the endomorphisms of trace zero. 

Denote by $\mathcal C(E)$ the space of holomorphic structures $\bar\partial$ on $E$ that induce the trivial holomorphic structure on $\det E\cong  \mathcal O$. It is an affine space for $\Omega^{0,1}(\Sigma,\sln(E))$. 
To formulate the self-duality equations, we must reduce the structure group to the maximal compact subgroup $\SU(n)$, i.e. we choose a hermitian metric $h$ on $E$. 
Then the self-duality equations are given by 
\begin{equation}\label{eq:sd}
\begin{aligned}
F^{\nabla_h}+[\Phi\wedge\Phi^{*_h}]&=0, \\
 \bar{\partial} \Phi&=0
 \end{aligned}
\end{equation}
for a holomorphic structure $\bar{\partial} \in\mathcal C(E)$ and $\Phi\in \Omega^{1,0}(\Sigma, \sln(E))$.
As usual, $\nabla_h$ is the Chern connection with respect to $\bar{\partial}$ and $h$.
Moreover, $*_h$ is the adjoint with respect to $h$, which we will sometimes just denote by $*$ if confusion is unlikely. 
We denote by 
\begin{equation*}
\mathcal{H}\subset \mathcal{C}(E)\times \Omega^{1,0}(\Sigma, \sln(E))
\end{equation*}
the space of solutions to \eqref{eq:sd}. 
Then the moduli space of such solutions is given by 
\[
\mathcal{M}_{SD}(\Sigma, \SU(n))=\mathcal{H}_{}/\mathcal{G}
\] 
with the special unitary gauge group $\mathcal{G}=\Gamma(\SU(E)) = \{g\in \Gamma(\mathrm{End}(E)\colon u^*u = \mathrm{id}, \det u  = 1\}$ acting by $(\bar\partial,\Phi).g = (g^{-1}\circ\bar\partial\circ g, g^{-1}\Phi g)$. The smooth locus of $\mathcal M_{SD}(\Sigma,\SU(n))$ is given by $\mathcal{M}_{SD}^{irr}(\Sigma,SU(n))=\mathcal H^{irr}/\mathcal G$, where $\mathcal H^{irr}$ denotes the set of irreducible solutions, i.e. those for which $(\overline{\partial},\Phi).g = (\overline{\partial},\Phi)$ implies that $g\in\mathcal G$ is a constant multiple of $\mathrm{id}_E$. 

The space $\mathcal{C}(E)\times \Omega^{1,0}(\Sigma, \mathfrak{sl}(E))$ is an affine space for $\Omega^{0,1}(\Sigma, \mathfrak{sl}(E))\oplus \Omega^{1,0}(\Sigma, \mathfrak{sl}(E))$. It carries a flat hyper-K\"ahler structure given by the three complex structures
\begin{equation}
\begin{aligned}
I_1(\gamma,\beta)&=(i\gamma,i\beta), \\
I_2(\gamma,\beta)&=(-\beta^*,\gamma^*),  \\
I_3(\gamma,\beta)&=(-i\beta^*, i\gamma^*),
\end{aligned}
\end{equation}
for $(\gamma,\beta)\in \Omega^{0,1}(\Sigma, \mathfrak{sl}(E))\oplus \Omega^{1,0}(\Sigma, \mathfrak{sl}(E))$ and metric 
\begin{equation}
 \|(\gamma,\beta)\|^2=2i \int_\Sigma \mathrm{tr}(\gamma^*\wedge \gamma+\beta\wedge \beta^*). 
\end{equation}
The holomorphic symplectic $2$-form $\omega_{\C} = \omega_2+i\omega_3$ (with respect to $I_1$) is 
\begin{equation}
\omega_{\C}((\gamma_1, \beta_1),(\gamma_2,\beta_2))= 2i \int_\Sigma \mathrm{tr}(\beta_2\wedge \gamma_1 - \beta_1 \wedge \gamma_2).
\end{equation}

The action of $\mathcal G$ preserves this hyper-K\"ahler structure and (formally) the self-duality equations are the vanishing condition for the associated hyper-K\"ahler moment map. Therefore, by the hyper-K\"ahler quotient construction, $\mathcal M_{SD}^{irr}(\Sigma,\SU(n))$ inherits a  hyper-K\"ahler structure \cite{Hitchin_1987, Fujiki}.

To give a complex analytic description of the twistor space of $\mathcal M_{SD}^{irr}(\Sigma,\SU(n))$ Deligne introduced the concept of a $\lambda$-connection \cite{Si-Hodge}:

\begin{definition}
Let $\lambda\in\C$.
A {\it holomorphic $\lambda$-connection} on $E$ is a pair $(\overline\partial,D)$ such that:
\begin{itemize}
\item $\overline\partial\in\mathcal C(E)$,
\item $D: \Gamma(E)\to \Omega^{1,0}(E)$ is a $C^\infty$ differential operator satisfying
\[
D(fs) \,=\, \lambda\partial f\otimes s + fDs,
\]
such that the induced differential operator on $\det E$ coincides with $\lambda \partial_\Sigma$.
\item The differential operator $D$ is holomorphic in the sense that 
\begin{equation}\label{dd-dd0}
D\overline{\partial}+\overline{\partial}D =0.
\end{equation}
\end{itemize}
The group of complex gauge transformations $\mathcal G_\C = \Gamma(\SL(E)) = \{g\in\Gamma(\End(E))\colon \det g \equiv 1\}$ acts on the space of $\lambda$-connections by 
\[
(\overline\partial,D).g = (g^{-1}\circ \overline{\partial}\circ g, g^{-1}\circ D \circ g).
\]

A holomorphic $\lambda$-connection $(\overline{\partial},D)$ on $E$ is called \emph{stable} (resp. semi-stable) if any $D$-invariant holomorphic subbundle $F\subset (E,\overline{\partial})$ satisfies $\deg F <0$ (resp. $\deg F\leq 0$). 
We call a holomorphic $\lambda$-connection \emph{polystable} if it is isomorphic to a direct sum of stable $\lambda$-connections whose associated holomorphic bundles have degree zero. 
\end{definition}

\begin{remark}\label{Rem: irred_lambda_conn}
The concept of holomorphic $\lambda$-connections gives a way of interpolating between flat $\SL(n,\C)$-connections and Higgs bundles.
\begin{enumerate}[(i)]
\item If $\lambda = 0$, then  $D$ is $C^\infty$-linear and holomorphic, hence defines a holomorphic section $\Phi\in H^0(\sln(E)\otimes K)$. Hence a $0$-connection is the same as an $\SL(n,\C)$-\emph{Higgs bundle}. The Higgs bundle is stable (resp. semi-stable, resp. polystable) in the sense of \cite{Hitchin_1987} if and only if the $0$-connection is stable (resp. semi-stable, resp. polystable) in the sense of the above definition. 
\item If $\lambda\neq 0$ and $(\overline{\partial}, D)$ is a holomorphic $\lambda$-connection, then the condition \eqref{dd-dd0} implies that we obtain a flat $\SL(n,\C)$-connection $\nabla$ via 
\[
\nabla = \overline{\partial} + \lambda^{-1}D.
\]
Stability in this case means that there exist no $\nabla$-invariant subbundles. A polystable $\lambda$-connection corresponds to a completely reducible flat connection, i.e. a direct sum of irreducible flat connections.
\item The action of the group of gauge transformations  specialises to the usual action on the space of Higgs bundles and flat $\SL(n,\C)$-connections, respectively.
\item A holomorphic $\lambda$-connection $(\overline{\partial},D)$ on $E$ is called {\it irreducible} if $(\overline\partial,D).g = (\overline\partial,D)$ implies that $g\in\mathcal G_\C$ is a constant multiple of the identity endomorphism $\mathrm{id}_E$, i.e. $g$ is a constant map to the center of $\SL(n,\C))$.
Since we work with vector bundles, irreducible $\lambda$-connections are equivalent to stable $\lambda$-connections.
In particular, irreducible $0$-connections are stable Higgs bundles. 
\end{enumerate}
\end{remark}

\begin{definition}
Let $\Sigma$ be a compact Riemann surface. The \emph{Hodge moduli space $\mathcal M_{Hod}(\Sigma,\SL(n,\C))$} is defined as
\[
\mathcal M_{Hod}(\Sigma,\SL(n,\C)) = \{(\overline{\partial},D,\lambda)\colon \lambda\in\C, \text{$(\overline{\partial},D)$ polystable hol. $\lambda$-connection}\}/\mathcal G_\C.
\]
\end{definition}
\begin{remark}
The Hodge moduli space $\mathcal{M}_{Hod}(\Sigma, \SL(n,C))$ is a complex space. 
Its smooth locus coincides with the locus $\mathcal{M}_{Hod}^{s}(\Sigma,\SL(n,\C))$ of stable $\lambda$-connections. 

Note that we have a holomorphic projection 
\[
p\colon\mathcal M_{Hod}(\Sigma,\SL(,\C))\to\C, \qquad (\overline{\partial}, D,\lambda)\mapsto \lambda.
\]
The map $(\overline{\partial},D,\lambda)\mapsto (\overline{\partial} + \lambda^{-1}D,\lambda)$ induces a biholomorphism
\[
p^{-1}(\C^*) \cong \mathcal M_{dR}(\Sigma, \SL(n,\C))\times\C^*,
\]
where $\mathcal M_{dR}(\Sigma, \SL(n,\C))$ is the moduli space of completely reducible flat $\SL(n,\C)$-connections on $\Sigma$. 
Via the Riemann-Hilbert correspondence, this is biholomorphic to the representation variety $\mathcal M_B(\Sigma, \SL(n,\C)) = \Hom^{red}(\pi_1(\Sigma),\SL(n,\C))\slash\SL(n,\C)$ of isomorphism classes of completely reducible representations of the fundamental group $\pi_1(\Sigma)$ into $\SL(n,\C)$.
\end{remark}

There is a natural isometric circle action on $\mathcal M_{SD}(\Sigma,\SU(n))$ induced by the circle action 
\[
e^{i\alpha} (\bar{\partial}, \Phi)=(\bar{\partial}, e^{i\alpha} \Phi)
\] 
on $\mathcal C(E)\times\Omega^{1,0}(\sln(E))$. It preserves the complex structure $I_1$ and rotates $I_2,I_3$. The action complexifies to a natural $\C^*$-action on $\mathcal M_{Hod}(\Sigma, \SL(n,\C))$ covering the standard $\C^*$-action on $\C$. A given $t\in\C^*$ acts on an element $(\overline\partial, D,\lambda)$ by 
\begin{equation}\label{Eq: C*action}
t.(\overline{\partial}, D, \lambda) = (\overline{\partial}, tD,t\lambda).
\end{equation}
\begin{definition}
Let $\Sigma$ be a compact Riemann surface and denote by $\overline \Sigma$ the conjugate surface. 
The Deligne--Hitchin moduli space $\mathcal M_{DH}(\Sigma, \SL(n,\C))$ is 
\[
\mathcal M_{DH}(\Sigma, \SL(n,\C)) = (\mathcal M_{Hod}(\Sigma, \SL(n,\C)) \, \dot\cup\, \mathcal M_{Hod}(\overline{\Sigma}, \SL(n,\C)))/\sim,
\]
where 
\[
(\overline{\partial},\,D,\,\lambda) \,\sim\, (\lambda^{-1} D, \,
\lambda^{-1}\overline{\partial}, \,\lambda^{-1})
\]
for any $(\overline{\partial},\,D,\,\lambda)\,\in\,
\mathcal M_{Hod}(\Sigma, \SL(n,\C))$ with $\lambda\,\neq\, 0$.
If we glue stable $\lambda$-connections, then we write $\mathcal{M}_{DH}^{s}(\Sigma,\SL(n,\C))$.
\end{definition}

\begin{remark}
The projections from the respective Hodge moduli spaces to $\C$ glue to give a holomorphic projection $\pi: \mathcal{M}_{DH}(\Sigma, \SL(n,\C))\to\mathbb CP^1$. 
The Deligne--Hitchin moduli space is a complex space.
Its smooth locus $\mathcal{M}^{s}_{Hod}(\Sigma,\SL(n,\C))$ coincides with the twistor space of $\mathcal{M}_{SD}^{irr}(\Sigma,\SU(n))$ (\cite[\S 4]{Si-Hodge}). 
\end{remark}
The anti-holomorphic involution $\tau$ can be seen via the Riemann--Hilbert correspondence as follows. On $\mathcal M_B(\Sigma,\SL(n,\C))$ we have the natural anti-holomorphic involution which associates to a representation $R:\pi_1(\Sigma)\to \SL(n,\C)$ its complex conjugate dual representation $\gamma\mapsto \overline{R(\gamma)^{-1}}^t$, i.e. $R$ is composed with the Cartan involution corresponding to the compact real form $\SU(n)$.  Under the Riemann--Hilbert correspondence $\mathcal M_B(\Sigma,\SL(n,\C))\cong \mathcal M_{dR}(\Sigma,\SL(n,\C))$ this induces an anti-holomorphic involution on the space of flat connections and we denote by $\overline{\nabla}^*$ the flat connection associated to $\nabla$ in this way. It can be interpreted as the connection on $\overline{E}^*$ induced by $\nabla$, hence the notation. We arrive at the following description of the anti-holomorphic involution on the Deligne--Hitchin moduli space (see also the discussion in \cite[\S 4]{Si-Hodge} and \cite[\S 1.4]{BHR}).
 
\begin{definition}\label{Def:AntiHolInv}
The Deligne--Hitchin moduli space comes with the following involutions. 
\begin{enumerate}[(i)]
\item We have the involution $N$ given by the action of $(-1)\in \C^*$:
\[
N\colon \mathcal M_{DH}(\Sigma, \SL(n,\C))\to\mathcal M_{DH}(\Sigma, \SL(n,\C)),\quad [(\overline{\partial},D,\lambda)]\mapsto [(\overline{\partial},-D,-\lambda)].
\]
\item We have an anti-holomorphic involution $\tau$, covering the antipodal involution $\lambda\mapsto -\overline{\lambda}^{-1}$.
\[
\tau\colon \mathcal{M}_{DH}(\Sigma, \SL(n,\C))\to\mathcal{M}_{DH}(\Sigma, \SL(n,\C)),\quad [(\overline{\partial},\,  D,\,\lambda)]\mapsto [(\overline{\lambda}^{-1}\overline{D}^*, -\overline{\lambda}^{-1}\overline{\overline\partial}^*, -\overline{\lambda}^{-1})] = [(\overline{\overline\partial}^*, -\overline{D}^*,-\bar\lambda)].
\]

\item From $N$ and $\tau$ we get a further anti-holomorphic involution $\rho = \tau\circ N = N\circ \tau\colon \mathcal{M}_{DH}(\Sigma, \SL(n,\C))\to\mathcal{M}_{DH}(\Sigma, \SL(n,\C))$ covering the inversion at the unit circle $ \lambda\mapsto \overline{\lambda}^{-1}$.
\end{enumerate}
\end{definition}

It is easily checked that the involutions $\tau$ and $\rho$ are compatible with the $\C^*$-action in the following sense. If $\sigma\in\{\rho,\tau\}$ and $t\in\C^*$, then  
\begin{equation}\label{Eq:Compatibility_sigma_C*action}
\sigma(t.(\bar\partial,D,\lambda)) = \overline t^{-1}.\sigma(\bar\partial,D,\lambda).
\end{equation}


\subsection{Sections of the Deligne--Hitchin moduli space}
In this subsection we recall some concepts and definitions from \cite{BHR}. 
\begin{definition} \label{Def: irred_section}
We call a holomorphic section $s\colon \mathbb CP^1\to\mathcal M_{DH}(\Sigma, \SL(n,\C))$ (i.e. a holomorphic map such that $\pi\circ s = \mathrm{id}_{\mathbb CP^1}$) \emph{irreducible} if the image of $s$ is contained in $\mathcal{M}_{DH}^{s}(\Sigma,\SL(n,\C))$.
\end{definition}
\begin{rem}
Note that we could also call such sections stable by Remark \ref{Rem: irred_lambda_conn}. 
\end{rem}
For every $k\in \mathbb{N}\cup \{\infty\}$, $\lambda$-connections of class $C^k$, instead of $C^\infty$, are defined in an obvious way. 
Also, the notion of holomorphic  $\lambda$-connections of class $C^k$ is defined correspondingly.
The next lemma shows their relation to (local) irreducible sections.

\begin{lemma}\cite[Lemma 2.2]{BHR}\label{lift-section}
Let $s\colon B\to \mathcal{M}_{DH}^{s}(\Sigma,\SL(n,\C))$ be an irreducible section where $B\subset \mathbb CP^1$ is an open neighborhood of $0\in \mathbb CP^1$.
For every $k\,\in\,\mathbb N^{\geq2}$, there exists a holomorphic lift 
\begin{equation} \label{eq:lift-section}
\widehat s(\lambda) =  (\overline{\partial}(\lambda), D(\lambda),\lambda) = (\overline\partial + \sum_{k=1}^\infty\lambda^k\Psi_k, \lambda\partial + \Phi + \sum_{k=2}^\infty\lambda^k\Phi_k,\lambda), \quad\lambda\in B'
\end{equation}
of $s$ to the space of holomorphic $\lambda$-connections of class $C^k$.
Here $B'\subset B$ is an open neighborhood of $0$ which equals $B$ if $B\subsetneq \mathbb CP^1$ and equals $\C$ if $B=\mathbb CP^1$.
In the latter case, there also exists a lift $^-\widehat s$ on $\mathbb CP^1\setminus\{0\}$. 
\end{lemma}

\begin{remark}\label{rem: lift-section-local}
The proof in \cite{BHR}	is formulated for $\SL(2,\C)$ and global irreducible sections but generalizes to the setup of Lemma \ref{lift-section}.
Note that if $B$ is sufficiently small, any irreducible section $s$ on $B$ admits a lift to the space of holomorphic $\lambda$-connections of class $C^\infty$. 
We lose regularity when such local lifts are glued together over larger $B$ though, see \cite{BHR} for details.

We further observe that if $s:B\to \mathcal{M}_{DH}(\Sigma, \SL(n,\C))$ is a local section around $0\in \mathbb{C}P^1$ such that $s(0)$ is a stable Higgs bundle, then there is an open neighborhood $B'\subset B$ of $0$ such that $s_{|B'}$ maps to $\mathcal{M}_{DH}^{s}(\Sigma,\SL(n,\C))$.
In particular, the germs of such sections always admit lifts to the space of holomorphic $\lambda$-connections. 

Finally, the above lemma applies for sections locally defined around $\infty\in\mathbb CP^1$ in the obvious way.
%
\end{remark}

Given an irreducible holomorphic section $s$ with lifts $\widehat{s},\, ^-\widehat{s}$ over $\C$ and $\mathbb CP^1\setminus\{0\}$ respectively, we will often work with the associated $\C^*$-family of flat connections
\[
^+\nabla = \bar\partial(\lambda) + \lambda^{-1}D(\lambda) = \lambda^{-1}\Phi +\nabla + ...
\]
and $^-\nabla$ defined similarly over $\C P^1\setminus\{0,\infty\}$. Here we write $\nabla = \bar\partial + \partial$ in the notation of equation \eqref{eq:lift-section}.
One can show \cite{BHR}, that there exists a holomorphic $\C^*$-family $g(\lambda)$ of $\mathrm{GL}(n,\C)$-valued gauge transformations, unique up to multiplication by a holomorphic scalar function,  such that $^+\nabla^\lambda.g(\lambda) = \ ^-\nabla^\lambda$. 

Irreducible holomorphic sections corresponding to solutions of the self-duality equations have the special property that we have lifts such that $^+\nabla = \ ^-\nabla$ on $\C^*$, i.e. we can arrange $g\equiv \mathrm{id}_E$ in the above discussion. This is axiomatised as follows.

\begin{definition}\label{admissible}
We call a holomorphic section $s$ of $\mathcal M_{DH}(\Sigma, \SL(n,\C))$ {\it admissible} if it admits a
lift $\widehat s$ on $\mathbb C$ of the form
\[{\widehat s}(\lambda)\,=\,(\overline{\partial}+\lambda \Psi,\,\lambda\partial+\Phi,\, \lambda)\]
for a Dolbeault operator $\overline\partial$ of type $(0,1)$, a
Dolbeault operator $\partial$ of type $(1,0)$, a $(1,0)$-form $\Phi$
and a $(0,1)$-form $\Psi$, such that $(\overline{\partial},\Phi)$ and $(\partial,\Psi)$ are  semi-stable Higgs pairs on $\Sigma$ respectively $\overline\Sigma$.
\end{definition}

In the rank $2$ case the family of gauge transformations $g(\lambda)$ such that $^+\nabla.g = \ ^-\nabla$ can be used to define the following invariant of an irreducible section $s$. 

\begin{definition}\label{Def: parity}
Let $s$ be an irreducible section of $\mathcal M_{DH}(\Sigma,\SL(2,\C))$ with associated families $^+\nabla$ and $^-\nabla$ over $\C$ and $\C P^1\setminus\{0\}$ respectively. Consider a holomorphic $\C^*$-family $g(\lambda)$ of $\mathrm{GL}(2,\C)$-valued gauge transformations such that $^+\nabla^\lambda.g(\lambda) = ^-\nabla^\lambda$. The \emph{parity of $s$} is the parity of the degree of the holomorphic function $\det g\colon \C^*\to\C^*$.
\end{definition}

\begin{remark}\label{rem:lift-section-admissible}
\begin{enumerate}[(i)]
\item We remark that the parity of an irreducible section $s$ is zero, if and only if we can arrange the family $g(\lambda)$ gauging $^+\nabla$ to $^-\nabla$ to be $\SL(2,\C)$-valued. In particular, any admissible section has parity zero.

In higher rank $n>2$ one can construct a similar invariant given by $\mathrm{deg}(\det g) \mod n\in\mathbb{Z}/n\mathbb{Z}$. For groups other than $\SL(n,\C)$ it is not obvious what an appropriate generalisation of this invariant might be.  
\item  Let $s$ be an irreducible holomorphic section which is not admissible. Consider the holomorphic family $g(\lambda)\colon  \Sigma\times\C^*\to\GL(2,\C)$ such that $^+\nabla.g(\lambda) = \, ^-\nabla$ 
and interpret this as a map $g\colon \Sigma\to\Lambda\GL(2,\C)$. Here $\Lambda\GL(2,\C)$ is (a suitable Sobolev completion of) the group of holomorphic maps $\C^*\to\GL(2,\C)$. Then it is shown in \cite[Proposition 2.7.]{BHR} that $g(\Sigma)\subset\Lambda\GL(2,\C)$ cannot be fully contained in the big cell of loops that admit a Birkhoff factorization of the form $g = g_+g_-$, where $g_+,g_-$ extends holomorphically to $\lambda =0$, $\lambda = \infty$ respectively. In other words, for a non-admissible section there is a non-empty subset $\gamma\subset \Sigma$ on which we cannot write $g = g_+g_-$. 
\end{enumerate}
\end{remark}

\begin{definition}\label{Def:sigmaReal}
Let $\sigma \in \{\tau,\rho\}$, where $\tau$ and $\rho$ denote the anti-holomorphic involutions defined in Definition \ref{Def:AntiHolInv}. A holomorphic section $s\colon \mathbb CP^1\longrightarrow\mathcal M_{DH}(\Sigma,\SL(n,\C))$ of the fibration $\pi\colon \mathcal M_{DH}(\Sigma,\SL(n,\C))\longrightarrow \mathbb CP^1$ is called {\it real with respect to $\sigma$}, or just {\it $\sigma$-real}, if 
$s(\lambda)\,=\, \sigma(s(\widetilde\sigma(\lambda)))$ for all $\lambda\,\in\, \mathbb CP^1$, where $\tilde\sigma\colon\mathbb CP^1\to\mathbb CP^1$ is the induced involution.
\end{definition}

If we have a lift $\nabla^\lambda$ on $\mathbb C\,\subset\,\mathbb CP^1$ of a $\sigma$-real holomorphic section $s$, then for every $\lambda\,\in\,\mathbb C^*$ there is a gauge
transformation $g(\lambda)$ such that the following equation holds
\begin{equation}\label{realeqsecsigma}
\nabla^\lambda.g(\lambda)\,=\,\overline{\nabla^{\widetilde{\sigma}(\lambda)}}^*.
\end{equation}

If $s$ is $\sigma$-real and irreducible of parity $0$, we can choose the family of gauge transformations $g(\lambda)$ in \eqref{realeqsecsigma} to depend holomorphically on $\lambda$ and may assume that it takes values in $\SL(2,\mathbb C)$. By irreducibility, the holomorphic family $g(\lambda)$ is then uniquely determined up to a sign.
By \cite[Lemma 2.15]{BHR} the following definition makes sense.

\begin{definition}\label{Def:RealPosNeg}
Let $\sigma\,\in\,\{\tau,\,\rho\}$ and consider an irreducible $\sigma$-real holomorphic section $s:\C P^1\to\mathcal M_{DH}(\Sigma,\SL(2,\C))$ of parity $0$. Let $\nabla^\lambda$ be a lift of $s$ over $\C$ and let $g(\lambda)$, $\lambda\,\in\,\C^*$, be a holomorphic family of $\SL(2,\mathbb C)$-valued gauge transformations such that 
\eqref{realeqsecsigma} holds. Then $s$ is called {\it $\sigma$-positive} if $-g(\lambda)\overline{g(\widetilde\sigma(\lambda))^{-1}}^t\,=\,{\rm Id}$ and {\it $\sigma$-negative} if $-g(\lambda)\overline{g(\widetilde\sigma(\lambda))^{-1}}^t\,=\,-{\rm Id}$.
\end{definition}

\begin{remark}\label{Rem: harmonicmaps}
\begin{enumerate}[(i)]
\item The signs in Definition \ref{Def:RealPosNeg} are chosen to be consistent with \cite{BHR}, where the fact that an $\SL(2,\C)$ bundle is isomorphic to its dual is incorporated into the definition, see \cite[\S 1.6]{BHR} for details.
\item If $(\bar\partial,\Phi)$ is a solution to the $\SU(2)$-self-duality equations, the associated twistor line is given by the $\C^*$-family of flat $\SL(2,\C)$-connections 
\[
\nabla^{\lambda} = \lambda^{-1}\Phi + \nabla_h + \lambda\Phi^{*_h}.
\]
It is shown in  \cite[Theorem 3.6]{BHR} that the irreducible solutions of the self-duality equations correspond precisely to the admissible, irreducible $\tau$-negative sections $\C P^1\to\mathcal M_{DH}(\Sigma,\SL(2,\C))$. By the non-abelian Hodge correspondence, these correspond to equivariant harmonic maps $f: \tilde\Sigma\to H^3 = \SL(2,\C)/\SU(2)$. 
\item On the other hand, $\rho$-negative  sections $\C P^1\to\mathcal M_{DH}(\Sigma,\SL(2,\C))$ are automatically admissible and correspond to equivariant harmonic maps $f:\tilde\Sigma\to S^3 = \SU(2)$. These are obtained from solutions to the harmonic map equations 
\begin{equation}\label{eq:hmS3}
\begin{aligned}
F^{\nabla_h}-[\Phi\wedge\Phi^{*_h}]&=0, \\
 \bar{\partial} \Phi&=0.
 \end{aligned}
\end{equation}
The associated sections are of the form $\nabla^\lambda = \lambda^{-1}\Phi + \nabla_h -\lambda\Phi^{*_h}$.
\end{enumerate}
\end{remark}

\subsection{Twisting}\label{Subsec:Twisting}
We briefly review the twisting or Gau\ss \ map procedure that played a central role in the construction of $\tau$-positive holomorphic sections of $\mathcal M_{DH}(\Sigma, \SL(2,\C))$ in \cite{BHR}. Starting from an irreducible solution $(\tilde\nabla,\tilde\Phi)$ to the $\SU(2)$-harmonic map equations \eqref{eq:hmS3} with nilpotent Higgs field, one considers the associated family of flat connections 
\[
\tilde \nabla^\lambda = \tilde \nabla + \lambda^{-1}\tilde \Phi - \lambda\tilde \Phi^*.
\]
Denote by $L$ the kernel bundle of $\tilde \Phi$, so that we get a smooth splitting $E = L\oplus L^\perp$. 
To this family $\tilde\nabla^\lambda$ of flat connections, one associates a new family $\nabla^\lambda$ of flat connections by \emph{twisting}, which is given by 
\[
\nabla^\lambda = \nabla^{\lambda^2}.h(\lambda),
\]
where 
\[
h(\lambda) = \left(\begin{array}{cc} \frac{1}{\sqrt{\lambda}} & 0\\ 0 & \sqrt{\lambda}\end{array}\right)
\]
is written with respect to the splitting $E=L\oplus L^\perp$. 
In \cite{BHR} it was shown that the so defined $\C^*$-family of flat connections extends to define an irreducible, admissible holomorphic section over all of $\mathbb CP^1$. In this section we study this procedure more systematically. 

Let $s\colon \mathbb CP^1\to\mathcal M_{DH}(\Sigma, \SL(n,\C))$ be a holomorphic section. Then we may use the $\C^*$-action to define a new holomorphic section $\tilde s$ over $\C^*$: 
\[
\tilde s: \C^*\to\mathcal M_{DH}(\Sigma, \SL(n,\C)),\qquad \tilde s(\lambda) = \lambda^{-1}.s(\lambda^2).
\]
Since the $\C^*$-action on $\mathcal M_{DH}(\Sigma, \SL(n,\C))$ covers the obvious one on $\mathbb CP^1 = \C\cup \{\infty\}$ it is clear that $\tilde s$ is a holomorphic section over $\C^*$.
It is a natural question to ask under what conditions on $s$ the twisted section $\tilde s$ extends to define a holomorphic section $\tilde s: \mathbb CP^1\to\mathcal M_{DH}(\Sigma, \SL(n,\C))$. 

\begin{definition}\label{Def:twistable}
We call a holomorphic section $s\colon \mathbb CP^1\to\mathcal M_{DH}(\Sigma,\SL(n,\C))$ \emph{twistable} if $\tilde s\colon  \C^*\to\mathcal M_{DH}(\Sigma, \SL(n,\C)), \lambda\mapsto \lambda^{-1}.s(\lambda^2)$, extends to a holomorphic section $\tilde s\colon \mathbb CP^1\to\mathcal M_{DH}(\Sigma, \SL(n,\C))$, which we call the \emph{twist} of $s$. 
\end{definition}

\begin{remark}\label{rem:twistlambdaconn}
In terms of $\lambda$-connections, we can view the construction of the twist as follows. Write 
\[
s(\lambda) = [(\overline\partial(\lambda),D(\lambda),\lambda)].
\]
Then for $\lambda\in\C^*$
\[
\tilde s(\lambda) = \lambda^{-1}.s(\lambda^2) = \lambda^{-1}.[(\overline \partial(\lambda^2), D(\lambda^2),\lambda^2)] = [(\overline \partial(\lambda^2), \lambda^{-1} D(\lambda^2), \lambda)].
\]
\end{remark}

The construction of \cite{BHR} suggests that there exists a transformation from the space of $\rho$-real twistable sections to the space of $\tau$-real sections. The precise result is as follows.

\begin{proposition}\label{prop:twistreal}
\begin{enumerate}[i)]
\item Suppose that $s\colon \mathbb CP^1\to \mathcal M_{DH}(\Sigma,\SL(n,\C))$ is a twistable holomorphic section. Then the twist $\tilde s$ is $N$-invariant.
\item Suppose that $s\colon \mathbb CP^1\to \mathcal M_{DH}(\Sigma,\SL(n,\C))$ is a $\rho$-real twistable holomorphic section. Then the twist $\tilde s$ is again $\rho$-real and moreover $N$-invariant, hence $\tau$-real.
\item Suppose that $s\colon \mathbb CP^1\to \mathcal M_{DH}(\Sigma,\SL(n,\C))$ is a $\tau$-real and $N$-invariant twistable holomorphic section. Then the twist $\tilde s$ is again $\tau$-real and moreover $N$-invariant.
\end{enumerate}
\end{proposition}
\begin{proof}
\begin{enumerate}[(i)]
\item We have by definition $\tilde s(\lambda) = \lambda^{-1}.s(\lambda)$ and so
\[
\tilde s(-\lambda) = (-\lambda^{-1}).s((-\lambda)^{2}) = (-\lambda)^{-1} s(\lambda^2) = (-1).\lambda^{-1}.s(\lambda^2) = N(\tilde s(\lambda)).
\]
\item We have $\tilde s(\lambda) = \lambda^{-1}.s(\lambda)$. Therefore
\[
\rho(\tilde s(\lambda)) = \rho(\lambda^{-1}.s(\lambda^2)) = \overline \lambda.\rho(s(\lambda^2)) = \overline\lambda.s(\overline \lambda^{-2}) = \tilde s(\overline \lambda^{-1})
\]
Thus,  using (i),
\[
\tau(\tilde s(\lambda)) = N(\rho (\tilde s(\lambda))) = N(\tilde s(\overline\lambda^{-1})) = \tilde s (-\overline\lambda^{-1}).
\]
\item Since the section $s$ is $\tau$-real and $N$-invariant, it must be $\rho$-real, since  
$$\rho(s(\lambda)) = N(\tau(s(\lambda))) = N(s(-\overline{\lambda}^{-1}) = s(\overline{\lambda}^{-1}). $$
By part (ii) the twist $\tilde s$ is therefore again $\tau$-real and $N$-invariant.
\end{enumerate}
\end{proof}

\begin{remark}
\begin{enumerate}[(i)]
\item In part (iii) of Proposition \ref{prop:twistreal} the assumption that the $\tau$-real section $s$ is moreover $N$-invariant is needed due to equation \eqref{Eq:Compatibility_sigma_C*action} with $\sigma = \tau$. In general we get $\tau(\tilde{s}(\lambda)) = \overline{\lambda}.s(-\overline{\lambda}^{-2})$ and the $N$-invariance then ensures the $\tau$-reality of $\tilde s$. 
\item Theorem 3.4 in \cite{BHR} can be interpreted as the statement that, in the $\SL(2,\C)$-case, an irreducible admissible $\rho$-negative section $s$ with nilpotent Higgs field is twistable and that the twist $\tilde s$ is $\tau$-positive.
\end{enumerate}
\end{remark}

The following proposition describes a class of twistable sections in the $\SL(2,\C)$-case. 

\begin{proposition}\label{twistablepro}
Let $s\colon \mathbb CP^1\to\mathcal M_{DH}(\Sigma,\SL(2,\C))$ be an irreducible holomorphic section such that the stable Higgs pairs $s(0) = (\overline{\partial},\Phi^+)$ and $s(\infty) = (\partial, \Phi^-)$ on $\Sigma$ and $\overline{\Sigma}$ have nilpotent Higgs fields. Then $s$ is twistable and the twist $\tilde s$ is an irreducible section of $\mathcal M_{DH}(\Sigma,\SL(2,\C))$. 
\end{proposition}
\begin{proof}
We need to prove that $\lambda\mapsto\tilde s(\lambda)$ extends to $\lambda=0,\infty$. Both cases work analogously,
so we only deal with $\lambda=0$.
Let us consider a lift $\nabla^\lambda$  of $s$ over $\{\lambda\neq\infty\}\subset\mathbb CP^1$ given by
\[
\nabla^\lambda = \lambda^{-1}\Phi^+ + \nabla + \sum_{k=1}^\infty\lambda^k\Psi_k.
\]
Here $\Phi^+\in\Omega^{1,0}$, and $\Psi_k\in\Omega^1$ for $k\geq 1$.
Then by Remark \ref{rem:twistlambdaconn} we get a lift of $\tilde s$ over $\C^*$ by
\[
\tilde\nabla^\lambda = \lambda^{-2}\Phi^+ + \nabla + \sum_{k=1}^\infty\lambda^{2k}\Psi_k.
\]
The section $\tilde s$ extends to $\lambda =0$ if we can find a $\C^*$-family $h(\lambda)$ of complex gauge transformations such that 
\[
\tilde\nabla^\lambda.h(\lambda) = \lambda^{-1}\tilde\Phi + \tilde\nabla+\sum_{k=1}^\infty\lambda^k\tilde\Psi_k
\]
and the Higgs pair $\tilde{s}(0) = (\bar\partial^{\tilde\nabla},\tilde\Phi)$ is stable. If $\Phi^+ = 0$ there is nothing to prove, so let us assume $\Phi^+\neq 0$. 
By assumption $\Phi^+$ is nilpotent, so let us denote by $L$ its kernel bundle, which must satisfy $\deg L<0$, since $(\bar\partial^\nabla,\Phi^+)$ is a stable Higgs pair by irreducibility of the section $s$ (see Definition \ref{Def: irred_section} and Remark \ref{Rem: irred_lambda_conn}). Take a complementary bundle $L^\perp$. Then, with respect to the splitting $E= L\oplus L^\perp$, we can take 
\[
h(\lambda) = \left(\begin{array}{cc} \lambda^{-1} & 0\\ 0 & 1\end{array}\right),
\]
and get with 
\[
\Phi^+ = \left(\begin{array}{cc} 0 & \phi\\ 0 &0\end{array}\right)
\]
the equation 
\[
\Phi^+.h(\lambda): = h(\lambda)^{-1}\Phi^+ h(\lambda) = \lambda\Phi^+.
\]
The flatness of $\nabla^\lambda$ implies $0=d^{\nabla}\Phi^+ = \overline\partial^{\nabla}\Phi^+$, so that $\nabla$ must be of the form
\[
\nabla = \left(\begin{array}{cc} \nabla^L & \alpha \\ \beta & \nabla^{L^\perp}\end{array}\right),
\]
with $\beta\in \Omega^{1,0}(\mathrm{Hom}(L,L^\perp))$. Then, writing 
\[
\Psi_1 = \left( \begin{array}{cc} \psi_{11} & \psi_{12} \\ \psi_{21} & \psi_{22}\end{array}\right),
\]
we get 
\[
\nabla.h(\lambda) = \left(\begin{array}{cc} \nabla^L & \lambda\alpha \\ \lambda^{-1}\beta & \nabla^{L^\perp}\end{array}\right), \qquad h(\lambda)^{-1} \Psi_1.h(\lambda) = \left( \begin{array}{cc} \psi_{11} &\lambda \psi_{12} \\ \lambda^{-1}\psi_{21} & \psi_{22}\end{array}\right).
\]
With this the lift $\tilde\nabla^\lambda$ transforms to 
\begin{eqnarray*}
\tilde\nabla^\lambda.h(\lambda) &=& \lambda^{-2}\Phi^+.h(\lambda) + \nabla.h(\lambda) + \lambda^2\Psi_1.h(\lambda) + \sum_{k=2}^\infty \lambda^{2k}\Psi_k.h(\lambda)\\
&=&  \lambda^{-1}\left(\begin{array}{cc} 0 & \phi\\ \beta &0\end{array}\right) + \left(\begin{array}{cc} \nabla^L & 0 \\ 0 & \nabla^{L^\perp}\end{array}\right) + \lambda\left(\begin{array}{cc} 0 & \alpha \\ \psi_{21}&0\end{array}\right) + \sum_{k=2}^\infty \lambda^k\tilde\Psi_k.
\end{eqnarray*}
It now follows just like in the proof of \cite[Theorem 3.4.]{BHR} that the section $\tilde s$ extends to $\lambda =0$ and that $\tilde s(0)$ is a stable Higgs pair. Moreover, for any $\lambda\neq 0$ the connection $\tilde\nabla^\lambda = \nabla^{\lambda^2}$ is irreducible, which implies that also $\tilde\nabla^\lambda.h(\lambda)$ is irreducible. Altogether this shows that $\tilde s$ is an irreducible section.
\end{proof}
We expect that a similar construction works for $n>2$ as well. 
The main difficulty is to verify the stability of $\tilde{s}(0)$ and $\tilde{s}(\infty)$ which is more involved for general $n>2$.

\section{The Energy Functional}\label{Sec:defene}
\subsection{The Definition of the Energy Functional} 
Consider a holomorphic section $s\colon  \C P^1\to\mathcal M_{DH}(\Sigma,\SL(n,\C)$. Assume that $s(0)$ is a stable Higgs pair. We shall denote the space of such sections by $\mathcal S$. By Remark \ref{rem: lift-section-local} there exists an neighbourhood $B$ of $0\in\C P^1$ and a lift 
\[
\widehat s(\lambda)=(\overline\partial+\lambda \Psi+\dots,\Phi+\lambda \partial+\dots,\lambda), \qquad \lambda\in B
\]
to the space of $\lambda$-connections with associated family of flat connections 
\[
\nabla^\lambda = \lambda^{-1}\Phi +\nabla + \lambda\Psi + \dots,\qquad \lambda\in B\setminus\{0\}.
\] 
Here $\Phi\in\Omega^{1,0}(\sln(E)), \Psi\in\Omega^{0,1}(\sln(E))$ and $\nabla = \bar\partial +\partial$ is an $\SL(n,\C)$-connection. 
We have seen in Lemma \ref{lift-section} that we can even find a global lift (i.e. $B = \C$) if the section $s$ is irreducible. 

Consider 
\begin{equation}\label{e1}
\mathbf{E}(\widehat s):=\frac{1}{2\pi i}\int_\Sigma \text{tr}(\Phi\wedge \Psi).
\end{equation}

\begin{proposition}\label{defene}
The quantity $\mathbf{E}(\widehat s)$ is independent of the choice of local lift  $\widehat s$ of the local section $s$. It defines a function $\mathcal E\colon\mathcal S\to\C$
\begin{gather*}
\mathcal E\colon \mathcal S\to\C\\
 \mathcal E(s):=\mathbf{E}(\widehat s).
 \end{gather*}
The function $\mathcal E\colon \mathcal S\to\C$ is holomorphic in the following sense: if $T$ is a complex manifold and $s\colon T\to\mathcal S, t\mapsto s_t$ is a holomorphic family of sections, then the function $\mathcal E\circ s\colon T\to\C, t\mapsto \mathcal E(s_t)$ is holomorphic.
\end{proposition}
\begin{proof}
Write $\widehat s = (\overline\partial+\lambda \Psi+\dots,\Phi+\lambda \partial+\dots,\lambda)$ as above. 
Let $\widehat s.g$ be another lift of $s$, where $g$ is a $\lambda$-dependent family of gauge transformations 
\[
g(\lambda)=g_0+\lambda g_1+..., \quad
\]
defined in a neighbourhoof of $0\in\C$.
We split $g(\lambda)$ into the product of a constant gauge transformation and a gauge transformation which equals the  identity for $\lambda=0$:
\[
g(\lambda)=g_0 (g_0^{-1} g(\lambda)).
\]
It is clear that $\mathbf{E}(\widehat s.g_0) = \mathbf{E}(\widehat s)$, since $\Phi$ and $\Psi$ are just conjugated by $g_0$. Thus, we may assume that $g_0 = 1$. Then 
\[
g(\lambda) = 1 + \lambda g_1 +\dots, \quad g^{-1}(\lambda) = 1 - \lambda g_1  + \dots\; .
\]
Thus,
\[
\widehat s(\lambda).g(\lambda) = (\bar\partial +\lambda(\Psi-\bar\partial g_1)+\dots, \Phi + \lambda\partial + \dots,\lambda)
\]
It then follows from Stokes' theorem and $\bar\partial \Phi = 0$ that
\[
\mathbf{E}(\widehat s.g(\lambda)) = \int_\Sigma \mathrm{tr}(\Phi \wedge (\Psi-\bar\partial g_1)) = \mathbf{E}(\widehat s).
\]
The holomorphicity of $\mathcal E$ as stated in the Proposition follows directly from the definition of $\mathcal E$. 
\end{proof}

\begin{definition}
We call $\mathcal E$ the \emph{energy functional} on the space $\mathcal S$ of holomorphic sections of $\mathcal M_{DH}(\Sigma, \SL(n,\C))$ admitting a local lift to the space of holomorphic $\lambda$-connections near $\lambda =0$. 
\end{definition}
\begin{remark}
\begin{enumerate}[(i)]
\item For the definition of $\mathcal E(s)$ we do not have to require the Higgs pair $s(0)$ to be stable. We have just made this assumption to streamline the exposition. The proof of Proposition \ref{defene} works without modification for any local section $s$ for which there exists a lift $\widehat{s}$ to the space $\lambda$-connections in a neighbourhood of $0\in\C P^1$. In particular we can define the energy of an admissible section, and in particular of a twistor line.
\item It is immediate from the definition that the energy functional $\mathcal E$ does only depend on the complex analytic structure
of the Deligne--Hitchin moduli space, and not on the identification of $\mathcal M_{DH}(\Sigma, \SL(n,\C))$ with the twistor space of the Higgs bundle moduli space.
\item The existence and relevance of such a functional is implicitly contained in 
\cite[Theorem 13.17]{HitchinHM} and  \cite[Theorem 9]{BabBob} for the case of tori, and \cite[Theorem 8]{He-spec} for holomorphic sections in certain
equivariant moduli spaces.
\end{enumerate}
\end{remark}
The name \emph{energy functional} is motivated by the following observation.
\begin{theorem}\label{harmonicenergy}
Let $s$ be a twistor line of ${\mathcal M}_{DH}(\Sigma,\SL(2,\C))\to\mathbb CP^1$ corresponding
to an equivariant harmonic map $f\colon \tilde\Sigma\to H^3 $, where $H^3$ is equipped with its constant sectional curvature $-1$ metric.
Then \[\mathcal E(s)=-\tfrac{1}{4\pi}\text{energy}(f),\]
where $\text{energy}(f)$ is the energy of $f$ on $\Sigma$. In particular, if $s$ is a twistor line, then $\mathcal E(s)\leq 0$.

Likewise, for a $\rho$-negative holomorphic section $s$ of ${\mathcal M}_{DH}(\Sigma,\SL(2,\C))\to\mathbb CP^1$ corresponding to an equivariant harmonic map
to $\SU(2)$ (equipped with its constant sectional curvature $1$ metric) we have
\[\mathcal E(s)=\tfrac{1}{4\pi}\text{energy}(f).\]
\end{theorem}
\begin{proof}
Take the associated family of flat connections, which provides us with a natural lift of such a ($\tau$- or $\rho$-)real holomorphic section (see Remark \ref{Rem: harmonicmaps}). The theorem then follows by interpreting the Higgs field as the $(1,0)$-part of the differential of the map $f$, see \cite{Donaldson-twisted,Hitchin_1987}.
\end{proof}

\begin{remark}
Analogous formulas hold for the case of equivariant harmonic maps of $\Sigma$ into the anti-de Sitter space $\SL(2,\R)$ and into the de Sitter space $\SL(2,\C)/\SL(2,\R)$.
\end{remark}

\begin{proposition}\label{Prop: SigmaEreal}
Let $\sigma\in\{\rho,\tau\}$ and let $s\colon \C P^1\to\mathcal M_{DH}(\Sigma,\SL(n,\C)$ be an \emph{admissible} holomorphic section. Then we have for the section $\sigma^* s = \sigma\circ s\circ\tilde\sigma$
$$\mathcal E(\sigma^*s) = \overline{\mathcal E(s)}.$$
In particular, if $s$ is $\sigma$-real, then its energy $\mathcal E(s)$ is real.
\end{proposition}
The last statement is generalized to arbitrary holomorphic $\tau$-real sections in Section \ref{hyperenergy} (cf. Lemma \ref{Lem:RealityOfResidue}) and Corollary \ref{cor:ResEnergy}). 
\begin{proof}
Since $s$ is admissible, we may find a global lift $\widehat s$ with associated $\C^*$-family $\nabla^\lambda$ of flat connections of the form
\[
\nabla^\lambda = \lambda^{-1}\Phi + \nabla +\lambda\Psi, 
\]
where $\Phi\in\Omega^{1,0}(\Sigma,\sln(E)), \Psi\in\Omega^{0,1}(\Sigma,\sln(E))$. Then $\sigma^*s = \sigma \circ s\circ\tilde\sigma$ has a lift given by 
$$^\sigma\nabla^\lambda = \overline{\nabla^{\tilde\sigma(\lambda)}}^* = -\overline{\tilde\sigma(\lambda^{-1})}\Phi^* + \overline{\nabla}^* - \overline{\sigma(\lambda)}\Psi^* = \begin{cases} \lambda^{-1}\Psi^* + \overline{\nabla}^* +\lambda\Phi^*, & \sigma = \tau\\ 
-\lambda^{-1}\Psi^* + \overline{\nabla}^* -\lambda\Phi^*, & \sigma = \rho
\end{cases}.$$
It follows that 
$$\mathcal E(\sigma^*s) = \frac{1}{2\pi i}\int_\Sigma \mathrm{tr}(\Psi^*\wedge\Phi^*) = -\frac{1}{2\pi i}\int_\Sigma \overline{\mathrm{tr}(\Phi\wedge\Psi)} =  \overline{\left(\frac{1}{2\pi i}\int_\Sigma\mathrm{tr}(\Phi\wedge\Psi)\right)} =  \overline{\mathcal E(s)}.$$
\end{proof}




\subsection{The effect of twisting on the energy}
In this paragraph we investigate how the energy functional behaves under the twisting construction introduced in Section \ref{Subsec:Twisting}.
 
\begin{proposition}\label{twisttheenergy}
Let $s\colon \mathbb CP^1\to\mathcal M_{DH}(\Sigma, \SL(2,\C))$ be an irreducible holomorphic section such that the stable Higgs pairs $s(0) = (\overline{\partial},\Phi^+)$ and $s(\infty) = (\partial, \Phi^-)$ on $\Sigma$ and $\overline{\Sigma}$ have nilpotent Higgs fields. Then the energy of the twisted section $\tilde s$ is given by
\[
\mathcal E(\tilde s) = 2\mathcal E(s) - \deg L,
\]
where $L$ is the kernel bundle of $\Phi^+$.
\end{proposition}
\begin{proof}
By Proposition \ref{twistablepro} we know that $s$ is twistable. 
With the same notation as in the proof of Proposition \ref{twistablepro} we have a lift of the form
\[
\tilde\nabla^\lambda =\nabla^{\lambda^2}.h(\lambda) = \lambda^{-1}\tilde\Phi +\tilde\nabla + \lambda\tilde\Psi_1 + \sum_{k=2}^\infty\lambda^k\tilde\Psi_k,
\]
where 
\[
\tilde\Phi = \left(\begin{array}{cc} 0 & \phi \\ \beta & 0\end{array}\right), \qquad \tilde\nabla = \left(\begin{array}{cc} \nabla^L & 0 \\ 0 & \nabla^{L^\perp}\end{array}\right), \qquad \tilde\Psi_1 = \left(\begin{array}{cc} 0 & \alpha \\ \psi_{21} & 0\end{array}\right).
\]
Let us for convenience relabel  $\psi = \psi_{21}$.
The energy of $s$ is given by 
\[
\mathcal E(s) = \mathbf{E}(\widehat s) =\frac{1}{2\pi i}\int_\Sigma \mathrm{tr}(\Phi\wedge\Psi) = \frac{1}{2\pi i}\int_\Sigma \phi\wedge\psi.
\]
The energy of the twist is 
\[
\mathcal E(\tilde s)= \mathbf{E}(\widehat{\tilde s}) = \frac{1}{2\pi i}\int_{\Sigma}\mathrm{tr}(\tilde\Phi\wedge\tilde\Psi) = \frac{1}{2\pi i}\int_\Sigma (\phi\wedge\psi + \beta\wedge\alpha) = \mathcal E(s) + \frac{1}{2\pi i}\int_{\Sigma}\beta\wedge\alpha.
\]
Since $\tilde\nabla^\lambda$ is flat for all $\lambda\in\C^*$, we see 
\begin{eqnarray*}
0 &=& F^{\tilde\nabla^\lambda} \\
&=&  \lambda^{-1}d^{\tilde\nabla}\tilde\Phi + \left(F^{\tilde\nabla} + [\tilde\Phi\wedge \tilde\Psi_1]\right) + \sum_{k=1}^{\infty}\lambda^k \left(d^{\tilde\nabla}\tilde\Psi_k + [\tilde\Phi\wedge\tilde\Psi_{k+1}] + \frac{1}{2}\sum_{j=1}^k [\tilde\Psi_j,\tilde\Psi_{k-j}]\right)\\
&=& \lambda^{-1}d^{\tilde\nabla}\tilde\Phi + \left(\begin{array}{cc} F^{\nabla^L} + \phi\wedge\psi + \alpha\wedge\beta & * \\ * & *\end{array}\right) + \sum_{k=1}^{\infty}\lambda^k \left(d^{\tilde\nabla}\tilde\Psi_k + [\tilde\Phi\wedge\tilde\Psi_{k+1}] + \frac{1}{2}\sum_{j=1}^k [\tilde\Psi_j,\tilde\Psi_{k-j}]\right).
\end{eqnarray*}
Thus, $F^{\nabla^L} + \phi\wedge\psi + \alpha\wedge\beta = 0$, i.e. 
\[
\beta\wedge\alpha = F^{\nabla^L} + \phi\wedge\psi.
\]
It follows that 
\[
\mathcal E(\tilde s) = \mathcal E(s) + \frac{1}{2\pi i}\int_{\Sigma}\beta\wedge\alpha = \mathcal E(s) + \frac{1}{2\pi i}\int_{\Sigma}(F^{\nabla^L} + \phi\wedge\psi) = 2\mathcal E(s) - \deg L.
\]
\end{proof}

\begin{corollary}\label{positivity}
Consider an irreducible, $\rho$-real, admissible holomorphic section $s\colon \C P^1\to\mathcal M_{DH}(\Sigma,\SL(2,\C))$ given by a family of flat connections 
\[
\nabla^\lambda = \lambda^{-1}\Phi + \nabla - \lambda\Phi^*,
\]
where $0\neq\Phi\in\Omega^{1,0}(\sln(E))$ is nilpotent and $(\nabla,\Phi)$ is an irreducible solution of the harmonic map equations \eqref{eq:hmS3}. Then the energy of the twisted section $\tilde s$ satisfies 
\[
\mathcal E(\tilde s) > |\deg L|>0,
\]
where $L$ denotes the kernel bundle of $\Phi$. In particular, the $\tau$-real section $\tilde s$ cannot be a twistor line.
\end{corollary}
\begin{proof}
By Theorem \ref{harmonicenergy} we know that $\mathcal E(s)> 0$, since $\Phi\neq 0$.
The energy of the twisted section $\tilde s$ is then (note that $\deg L<0$ by irreducibility)
\[
\mathcal E(\tilde s) = 2\mathcal E(s) -\deg L> |\deg L|>0.
\]
We know from Proposition \ref{prop:twistreal} that $\tilde s$ is $\tau$-real since $s$ is $\rho$-real. But for a twistor line the energy is negative, while we have just seen that $\mathcal E(\tilde s)>0$. Hence $\tilde s$ cannot be a twistor line.
\end{proof}

\begin{remark}
The proof of \cite[Theorem 3.4]{BHR} has two steps: first an irreducible $\rho$-real section $s$ with nilpotent Higgs field corresponding to a solution of the equation \eqref{eq:hmS3} is constructed. Then it is shown that the twist $\tilde s$ is an admissible $\tau$-positive section, and therefore cannot be a twistor line. The above corollary simplifies the second step by showing instead $\mathcal E(\tilde s)>|\mathrm{deg}(L)|>0$ and applying Theorem \ref{harmonicenergy}. 
\end{remark}

\section{Interpretation of the Energy via the hyperholomorphic Line bundle}\label{hyperenergy}
We have defined the energy functional on the space of holomorphic sections of the Deligne--Hitchin moduli space $\mathcal{M}^{irr}_{DH}\to \mathbb CP^1$. In this section we will see that it is in fact a natural functional on the space of holomorphic sections of the twistor space $Z(M)\to \mathbb CP^1$ of any hyper-K\"ahler manifold $M$ with a circle action that induces a standard rotation of the $S^2$ of complex structures. 
The crucial tool is the hyperholomorphic line bundle introduced by Haydys \cite{Haydys} which was given a twistorial description by Hitchin \cite{Hitchin-hkqk}.   

More precisely, let $(M^{4k};g,I_1,I_2,I_3)$ be a hyper-K\"ahler manifold with corresponding K\"ahler forms $\omega_j$, $j=1,2,3$. 
Suppose that $M$ comes equipped with an isometric circle action which preserves $\omega_1$ and rotates $\omega_2$, $\omega_3$, i.e. 
\begin{equation}\label{eq:relations}
\mathcal{L}_X\omega_1=\omega_1, \quad \mathcal{L}_X\omega_2=-\omega_3, \quad \mathcal{L}_X\omega_3=\omega_2,
\end{equation}
where $X$ is the vector field induced by the circle action. 

Let $\mu\colon M\to i\R$ be the moment map corresponding to $\omega_1$. 
Haydys has shown that the two-form\footnote{As usual, we write $d_j^c=I_j\circ d \circ I_j$ for $j=1,2,3$. Also note that both Haydys and Hitchin work with $\R$-valued moment maps explaining our additional factor $i$, for example $i\, d\mu= i_X\omega_1$.} $\omega_1+idd_1^c \mu$ is of type $(1,1)$ with respect to each $I_j$, $j=1,2,3$.
In particular, if $[\omega_1/2\pi]\in H^2(M,\Z)$, then there exists a line bundle $L_M\to M$ such that $c_1(L_M)=[\omega_1/2\pi]$ which carries a unitary connection with curvature equal to $\omega_1+idd_1^c \mu$. Such line bundles are unique up to tensor product with a flat line bundle on $M$ and we fix one in the following. 
By Haydys' result this connection is hyperholomorphic: it induces a holomorphic structure on $L$ with respect to each complex structure in the $S^2$-family of compatible complex structures.
The hyperholomorphic line bundle $L_M$ induces a holomorphic line bundle $L_Z\to Z$ on the twistor space $\pi_Z\colon Z=Z(M)\to \mathbb CP^1$ of $M$, which is trivial on the twistor lines.
\begin{rem}\label{rem:iXomega1sufficient}
	We emphasize that for the existence of $L_M$, and hence $L_Z$, it is sufficient to work with $i_X\omega_1$ instead of the moment map (also see \cite{Korman}).
\end{rem}

\subsection{Meromorphic connections on $L_Z$}
Before we define the energy functional in general, we need some background on meromorphic connections on $L_Z$ complementing the results of \cite{Hitchin-hkqk}. 
The line bundle $L_Z$ corresponds to the Lie algebroid
\begin{equation}\label{eq:liealg}
  \begin{tikzcd}
    0 \ar[r] & \mathcal{O}_Z \ar[r] & TP_Z/\C^* \ar[r] & T_Z \ar[r] & 0
  \end{tikzcd}
\end{equation}
where $P_Z$ is the principal $\C^*$-bundle corresponding to $L_Z$ and $TP_Z/\C^*$ is the vector bundle on $Z$ whose sections correspond to the $\C^*$-invariant vector fields on $P_Z$. 
For later reference we denote the extension class of \eqref{eq:liealg} by $\eta_Z\in H^1(Z,T_Z^*)$. Since $TP_Z/\C^*$ is a Lie algebroid, $\eta_Z$ actually lies in (the image of) $H^1(Z,T_{Z,cl}^*)$ for the closed $1$-forms $T_{Z,cl}^*$. That is, $\eta_Z$ can be represented by a \v{C}ech cocyle with values in the sheaf of closed one-forms.

Hitchin observed that $\eta_Z$ is of a special form if additionally $H^1(M,\C)=0$.
Namely, let $Y\in H^0(Z,T_Z)$ be the holomorphic vector field induced by the circle action lifted to the twistor space $Z$.
After applying M\"obius transformations, we may assume that the circle action satisfies
\begin{equation}
  d\pi_Z(Y)=s:=\pi_Z^*\left(i\lambda \frac{d}{d\lambda}\right)\in H^0(Z,\pi^*\mathcal{O}(2)).
\end{equation}
If $D:=D_0+D_\infty$ denotes the divisor determined by the fibers of $\pi_Z$ over $0$ and $\infty$, then $s$ yields the short exact sequence
\begin{equation}\label{eq:T2}
  \begin{tikzcd}
    0 \ar[r] & T_Z^* \ar[r, "s"] & T_Z^*(2) \ar[r, ""] & T_Z^*(2)_{|D} \ar[r] & 0.
  \end{tikzcd}
\end{equation}
In his twistorial approach to $L_Z$, Hitchin constructed a section $\varphi\in H^0(D,T_Z^*(2)_{|D})$, assuming $H^1(M,\C)=0$, which satisfies
\begin{equation}\label{eq:condphi}
  \delta_Z(\varphi)=\eta_Z, \quad \varphi_{|T_F}=-\tfrac{1}{2}i_Y\omega_{|D}.
\end{equation}
Here $\delta_Z\colon H^0(D,T_Z^*(2)_{|D})\to H^1(Z,T_Z^*)$ is the connecting homomorphism in the long exact sequence associated to \eqref{eq:T2}, $T_F = \ker d\pi_Z$ is the tangent bundle along the fibers of $\pi_Z$ and 
\begin{equation}\label{eq:omega}
\omega=(\omega_2+i\omega_3)+2i\lambda \omega_1+\lambda^2(\omega_2-i\omega_3) \in H^0(Z,\Lambda^2 T_Z^*(2)).
\end{equation}
We next show that $\varphi$ satisfying \eqref{eq:condphi} is essentially unique. 

\begin{prop}\label{p:mero}
  Let $M$ be a connected hyper-K\"ahler manifold with a circle action as before.
  Additionally assume that $[\omega_1/2\pi]\in H^2(M,\Z)$.
  Then the space 
  \begin{equation*}
    \{\varphi \in H^0(D,T_Z^*(2)|_D)~|~\varphi \mbox{ satisfies } \eqref{eq:condphi} \}
  \end{equation*}
  is an affine complex line if non-empty (e.g. if $H^1(M,\C)=0$). 
Each such $\varphi$ determines a \emph{unique} meromorphic connection $\nabla_\varphi$ on $L_Z$ with simple poles along $D$ with
  \begin{equation}\label{eq:res}
    \mathrm{res}_D(\nabla_\varphi)=\varphi \in H^0(D,T_Z^*(2)|_D)
  \end{equation}
 and non-singular otherwise. 
\end{prop}
The existence of a meromorphic connection $\nabla_\varphi$ with \eqref{eq:res} already appeared in \cite{Hitchin-hkqk} but we include its proof for completeness. 

As a preparation we give the proof of the following well-known lemma. 
\begin{lem}
Let $\pi_Z\colon Z\to \mathbb CP^1$ be the twistor space of a connected hyper-K\"ahler manifold $M$. 
Then $H^0(Z,\Lambda^k T_Z^*)=0=H^0(Z,\Lambda^k T_F^*)$ for all $k\geq 1$ and in particular $H^0(Z,\mathcal{O}_Z)=\C$. 
\end{lem}
\begin{proof}
Let $\alpha\in H^0(Z,T_Z^*)$ and let $s_m\colon \mathbb CP^1\to Z$ be the twistor line determined by $m\in M$.
Then $s_m$ induces the holomorphic splitting
  \begin{equation*}
  s_m^*T_Z^*\cong s_m^*T_F^*\oplus s_m^*(\pi_Z^*\mathcal{O}(-2))\cong N_{s_m}\oplus \mathcal{O}(-2) \cong \mathcal{O}(-1)^{\oplus d} \oplus \mathcal{O}(-2).
  \end{equation*}
  Hence if we restrict $\alpha$ (as a section) to the image $s_m(\mathbb CP^1)\subset Z$, we obtain $\alpha_{|s_m(\mathbb CP^1)}=0$. 
  By varying $m\in M$, we conclude $\alpha=0$. 
  The same argument shows $H^0(Z, \Lambda^k T_Z^*)=0=H^0(Z,\Lambda^k T_F^*)$. 
  Since $M$ is connected, it immediately follows that $H^0(Z,\mathcal{O}_Z)=\C$. 
\end{proof}

\begin{proof}[Proof of Proposition \ref{p:mero}]
  First of all, we consider for each $\mathcal{F}\in \left\{T_Z^*, T_F^*,\pi_Z^*T_{\mathbb CP^1}^*=\pi_Z^*\mathcal{O}(-2)\right\}$ the short exact sequence
  \begin{equation}\label{eq:ShortExactSequenceOfComplexes}
    \begin{tikzcd}
      C^\bullet(\mathcal{F}): & 0 \ar[r] & \mathcal{F} \ar[r, "s"] & \mathcal{F}(2) \ar[r] & \mathcal{F}(2)_{|D} \ar[r] & 0.
    \end{tikzcd}
  \end{equation} 
  These fit into the diagram
  \begin{equation}\label{eq:T3}
    \begin{tikzcd}
      0 \ar[d] \\
      C^\bullet(\pi_Z^*\mathcal{O}(-2)) \ar[d] \\
      C^\bullet(T_Z^*) \ar[d] \\
      C^\bullet(T_F^*) \ar[d] \\
      0
    \end{tikzcd}
  \end{equation}
  with exact rows and columns.
  Next we consider (parts of) the corresponding long exact sequences.
  Since $H^0(Z,T_Z^*)=0=H^0(Z,T_F^*)$ and $H^0(Z,\mathcal{O}_Z)=\C$, we obtain the following diagram\footnote{We drop the spaces from the notation to simplify notation.} 
  \begin{equation}\label{eq:les}
    \begin{tikzcd}
        \C \ar[r, hookrightarrow] \ar[d, hookrightarrow] & H^0(\mathcal{O}_D) \ar[r] \ar[d, hookrightarrow] & H^1(\pi_Z^*\mathcal{O}(-2)) \ar[d, hookrightarrow] \\
        H^0(T_Z^*(2)) \ar[r, hookrightarrow] \ar[d] & H^0(T_Z^*(2)_{|D}) \ar[r, "\delta_Z"] \ar[d, "r"] & H^1(T_Z^*) \ar[d] \\
        H^0(T_F^*(2)) \ar[r, hookrightarrow]  & H^0(T_F^*(2)_{|D}) \ar[r] & H^1(T_F^*) 
    \end{tikzcd}
  \end{equation}
  Now let $\varphi,\varphi'\in H^0(T_Z^*(2)_{|D})$ satisfy \eqref{eq:condphi} and set $c:=\varphi-\varphi'$. 
  By assumption, we have $\delta_Z(c)=0$ and $r(c)=0$. 
  Therefore the exactness of \eqref{eq:les} implies that $c\in H^0(Z,\mathcal O_Z)=\C$. 
  To make this more explicit for later purposes, note that $\mathcal{O}_D$ appearing in \eqref{eq:les} is actually $\pi^*\left(\mathcal{O}(-2)\otimes \mathcal{O}(2)\right)_{|D}$. 
  Using $\gamma:=\pi_Z^*d\lambda\otimes \pi_Z^* \tfrac{\partial}{\partial \lambda}$ as trivializing section\footnote{Note that we pullback $d\lambda$ as a $1$-form but $\tfrac{\partial}{\partial \lambda}$ as a section.}, we have 
  \begin{equation}\label{eq:triv}
   c\mapsto (c\, \gamma_{|D_0}, c \, \gamma_{|D_\infty})\in H^0(D_0, T_Z^*(2)_{|D_0})\oplus H^0(D_\infty, T_Z^*(2)_{|D_\infty})
  \end{equation}
  under the inclusion $\C\hookrightarrow H^0(T_Z^*(2)_{|D})$ in \eqref{eq:les}. 
  In that sense $\varphi\in H^0(Z,T_Z^*(2)_{|_D})$ satisfying \eqref{eq:condphi} is unique up to an additive constant. 
  
For the last statement, assume $\varphi\in H^0(D, T_Z^*(2)|_D)$ with \eqref{eq:condphi} exists. This is the case, for example, if $H^1(M,\C)=0$. 
Then choose an appropriate open covering $\mathcal{U}$ of $Z$ such that $\varphi_{|U\cap D}$ lifts to $\varphi_U\in H^0(U,T_U^*(2))$ for every $U\in \mathcal{U}$.
By \eqref{eq:condphi}, the cocycle
\begin{equation*}
 \eta_{UV}=\tfrac{\varphi_U}{s}-\tfrac{\varphi_V}{s}
\end{equation*}
is a \v{C}ech cocycle representing $\eta_Z = \delta_Z(\varphi)\in H^1(Z,T_Z^*)$. 
But $\eta_Z$ is determined by the line bundle $L_Z$ so that $\eta_{UV}=g_{UV}^{-1} dg_{UV}$ for a \v Cech cocyle $g_{UV}$ defining $L_Z$. 
Altogether  
\begin{equation*}
g_{UV}^{-1}dg_{UV}= \eta_{UV}= \tfrac{\varphi_U}{s} -\tfrac{\varphi_V}{s}
\end{equation*}
so that $\tfrac{\varphi_U}{s}$ are connection $1$-forms of a meromorphic connection $\nabla_\varphi$ on $L_Z$ with the claimed properties. 

 For the uniqueness of $\nabla_\varphi$, let $\nabla_1, \nabla_2$ be two meromorphic connections on $L_Z$ with $\mathrm{res}_D(\nabla_1)=\mathrm{res}_D(\nabla_2)$ and holomorphic otherwise. 
  Then $\nabla:=\nabla_1\otimes \nabla_2^*$ is a holomorphic connection on $L_Z\otimes L_Z^*=\mathcal{O}_Z$.
  Hence $\nabla$ is of the form $d+\alpha$ for a global holomorphic $1$-form $\alpha$ on $Z$ which must vanish. 
  Consequently, $\nabla_1$ and $\nabla_2$ are equal.
\end{proof}

As a next step, we examine how such $\varphi$ interact with the real structure $\tau_Z$.
First observe that for every $\varphi$ satisfying \eqref{eq:condphi}, the section 
\begin{equation*}
\varphi^r:=\tfrac{1}{2}(\varphi+\overline{\tau_Z^*\varphi}) 
\end{equation*}
again satisfies \eqref{eq:condphi} and is further real, i.e. 
\begin{equation}\label{eq:phireal} 
 \tau_Z^*(\varphi^r)=\overline{\varphi^r} \quad \Leftrightarrow \quad\tau_Z^*\varphi^r_0=\overline{\varphi^r_\infty}. 
\end{equation}
Since $\delta_Z$ commutes with $\overline{\tau_Z}^*$, it follows that $\tau_Z^*(\eta_Z)=\overline{\eta}_Z$ and consequently
$
 \tau_Z^*L_Z\cong \overline{L}_Z. 
$
\begin{rem}
Hitchin's sections $\varphi\in H^0(D, T_Z^*(2)|_D)$ satisfying \eqref{eq:condphi} are real by construction. 
\end{rem}
For later reference, we record the following observation.
\begin{cor}\label{cor:real} 
The space 
\begin{equation*}
 \{ \varphi\in H^0(D,T_Z^*(2)|_D)~|~\varphi \mbox{ satisfies } \eqref{eq:condphi} \mbox{ and }\tau_Z^*\varphi=\overline{\varphi} \} 
\end{equation*}
is an affine real line (if non-empty).
\end{cor} 
\begin{proof}
By Proposition \ref{p:mero}, if $\varphi, \varphi'$ obey \eqref{eq:condphi}, then $\varphi-\varphi'=c\in H^0(Z, \mathcal{O})=\C$. 
The reality condition implies 
\begin{equation*}
 c=\tau_Z^*(\varphi-\varphi')=\overline{\varphi}-\overline{\varphi'}= \overline{c}. 
\end{equation*}
\end{proof}

\subsection{Residues}
For the next proposition, we assume $H^1(M,\C)=0$ so that real sections $\varphi\in H^0(D,T_Z^*(2))$ as in Corollary \ref{cor:real} and correspondingly meromorphic connections $\nabla_{\varphi}$ on $L_Z$ exist as in Proposition \ref{p:mero}. 
Let $\mathcal{S}$ be the complex-analytic space of holomorphic sections of $\pi_Z\colon Z\to \mathbb CP^1$ and define the function
\begin{equation*}
 \mathrm{res}_\varphi: \mathcal{S}\to \C, \quad \mathrm{res}_\varphi(s):=\mathrm{res}_0(s^*\nabla_\varphi). 
\end{equation*}
This is well-defined because $s^*\varphi_0\in H^0(\{0\}, \mathcal{O})=\C$. 
It is immediate that $\mathrm{res}_\varphi$ is holomorphic in the following sense: 
if $T$ is a complex manifold and $(s_t\colon\mathbb CP^1\to Z)_{t\in T}$ a holomorphic family of sections of $\pi_Z$, then 
\begin{equation*}
T\to \C, \quad t\mapsto \mathrm{res}_\varphi(s_t) 
\end{equation*}
is holomorphic. 
We further observe that $\mathrm{res}_\varphi(s)$ is defined for any local holomorphic section around $0\in \mathbb CP^1$. 

In case $s$ is a real section defined on all of $\mathbb CP^1$, then we obtain the following relation:
\begin{lem}\label{Lem:RealityOfResidue}
	If $s\in \mathcal{S}^{\R}$ is a real holomorphic section of $\pi_Z$, then 
	\begin{equation}\label{eq:ConjugateOfRes}
		\overline{\mathrm{res}_\varphi(s)}=\mathrm{deg}(s^*L_Z)-\mathrm{res}_\varphi(s).
	\end{equation}
	In particular, if $s$ is a real holomorphic section with $\deg(s^*L_Z)=0$, then $\mathrm{res}_\varphi(s)\in i\R$. 
\end{lem}
\begin{proof}
	Since $\overline{\varphi_0}=\tau_Z^*\varphi_\infty$, we obtain by the reality of $s$
	\begin{equation*}
		\overline{\mathrm{res}_\varphi(s)}=\overline{s^*\varphi_0}=s^*\tau_Z^*\varphi_\infty=\tau_{\mathbb CP^1}^*s^*\varphi_\infty=\mathrm{res}_\infty(s^*\nabla_\varphi).
	\end{equation*}
	The last equation uses again the fact that $H^0(\{\infty\},\C)=\C$ canonically so that the pullback along $\tau_{\mathbb CP^1}$ has no effect. 
	By the residue formula for meromorphic connections on line bundles over Riemann surfaces, we have
	\begin{equation}\label{eq:ResidueAt0AndInfty}
	\mathrm{deg}(s^*L_Z)=\mathrm{res}_0(s^*\nabla_\varphi)+\mathrm{res}_\infty(s^*\nabla_\varphi).
	\end{equation}
	Combining the two formulas, we arrive at \eqref{eq:ConjugateOfRes}.
\end{proof}
The previous lemma reflects the fact that $\mathrm{res}_\varphi$ yields a moment map on all connected components of $\mathcal{S}^\R$, see Section \ref{s:EnergyAsMomentMap}.
To show this and the relation of $\mathrm{res}_\varphi$ to the previously defined energy functional, we need an explicit formula for $\mathrm{res}_\varphi$. 
We begin with the following lemma (see \cite[Lemma 8]{Hitchin-hkqk}).
\begin{lem}\label{l:psi}
Let $T_Z^*(2)=T_F^*(2)\oplus \mathcal{O}_Z$ be the $C^\infty$-splitting induced by the $C^\infty$-decomposition $Z=M\times \mathbb CP^1$. 
Define the section $\psi\in \Gamma(Z,T_F^*(2)\oplus \mathcal{O}_Z)$ via 
\begin{equation}\label{eq:psi}
\psi=(\tfrac{\phi}{2}, \mu),\quad \phi=(d_2^c\mu+id_3^c\mu)+2i\lambda d_1^c\mu+\lambda^2(d_2^c\mu- i d_3^c\mu).
\end{equation} 
Then $\psi_{|D}$ is a holomorphic section of $T_Z(2)_{|D}$ and satisfies (\ref{eq:condphi}).
\end{lem}
\begin{proof}
Let $X$ be the $C^\infty$-vector field on $M$ induced by the circle action which we identify with a $C^\infty$-vector field on $Z=M\times \mathbb CP^1$. 
We denote by $X^{1,0}_{\lambda}$ the $(1,0)$-part of $X$ with respect to the complex structure $I_\lambda$ on $M$. 
The holomorphic structure $\bar{\partial}_Z$ on $T_Z^*(2)=T^*_F(2)\oplus \pi_Z^*\mathcal{O}_{\mathbb CP^1}$ (with respect to the natural $C^\infty$-splitting) is given by 
\begin{equation}\label{eq:csTZ}
\bar{\partial}_Z(\alpha,u)=\left(\bar{\partial}_\lambda \alpha, \bar{\partial}_\lambda u+i \lambda^{-1}\alpha\left(\bar{\partial}_\lambda X_\lambda^{1,0}\right)\right), 
\end{equation}
cf. equation (8) after Lemma 7 in \cite{Hitchin-hkqk}. 
Here we abuse notation and write $\bar{\partial}_\lambda$ also for the induced complex structure on $\mathcal{O}(2)$-valued one-forms etc.


Let $F = \omega_1 + i dd_1^c\mu$, the curvature of the hyperholomorphic line bundle on $M$.
We first prove  
\begin{equation}\label{eq:ppsi}
\bar{\partial}_Z \psi=(\lambda F,0)
\end{equation}
which implies that $\psi_{|D}$ is holomorphic and satisfies the first equation in (\ref{eq:condphi}), i.e. determines the hyperholomorphic line bundle. 
The first summand in (\ref{eq:ppsi}) is $\bar{\partial}_\lambda (\tfrac{\phi}{2})$. 
Since $\phi$ is of type $(1,0)$ with respect to $\bar{\partial}_\lambda$ for all $\lambda$, it follows that $\bar{\partial}_\lambda \phi=(d\phi)^{1,1}$. 
But $dd_k^c\mu= i\,\omega_k$, $k=2,3$ so that 
\begin{equation}
d\phi= 2i \lambda dd_1^c\mu+ i\, (\omega_2+i\omega_3)+i\, \lambda^2 (\omega_2-i\omega_3).
\end{equation}
Since $\omega =(\omega_2+i\omega_3)+2i\lambda \omega_1+\lambda^2(\omega_2-i\omega_3)$ as in (\ref{eq:omega}) is of type $(2,0)$ with respect to $\bar{\partial}_\lambda$ for every $\lambda$, we conclude 
\begin{equation*}
 \bar{\partial}_\lambda \left(\tfrac{\phi}{2}\right)=\tfrac{1}{2}(d\phi)^{1,1}=i \lambda (dd_1^c\mu -i\, \omega_1)^{1,1}=\lambda F. 
\end{equation*}
For the second summand in (\ref{eq:ppsi}), we have to show 
\begin{equation}\label{eq:pmu}
2i\lambda\bar{\partial}_\lambda\mu=\phi\left(\bar{\partial}_\lambda X^{1,0}_\lambda\right)
\end{equation}
by (\ref{eq:csTZ}). 
Since $\phi$ is of type $(1,0)$ with respect to $\bar{\partial}_\lambda$, we have $i_X\phi=i_{X^{1,0}_\lambda}\phi$ and therefore 
\begin{equation*}
\phi\left(\bar{\partial}_\lambda X^{1,0}_\lambda\right) =\bar{\partial}_\lambda (i_X\phi)+i_{X^{1,0}_\lambda} \bar{\partial}_\lambda \phi=\bar{\partial}_\lambda (i_X\phi)+ 2\lambda (i_XF)^{0,1}
\end{equation*}
from (\ref{eq:ppsi}).
To evaluate $i_X\phi$, observe that 
\begin{equation*}
 i_X d_1^c\mu =-i\, g(X,X),\quad i_Xd_2^c\mu=-i\, \omega_3(X,X)=0, \quad i_Xd_3^c\mu=i\,\omega_2(X,X)=0
\end{equation*}
and hence $i_X\phi=2\lambda g(X,X)$. 
Further $(i_X F)^{0,1}= i\,\bar{\partial}_\lambda (\mu+i \,g(X,X))$ so that
\begin{equation}
\phi(\bar{\partial}_\lambda X^{1,0}_\lambda)= 2i\lambda\, \bar{\partial}_\lambda \mu
\end{equation}
and \eqref{eq:pmu} is proven. 

It remains to prove that $\psi$ satisfies $\psi_{|T_F}=\tfrac{1}{2i } i_Y \omega$ along $D$ (again for $\omega$ as in (\ref{eq:omega})). 
Since $\psi$ is real, it suffices to check this equality at $\lambda=0$: 
\begin{equation*}
\psi_{|T_{D_0}}=\tfrac{\phi}{2}_{|\lambda=0} = \tfrac{1}{2}(d_2^c\mu+i d_3^c\mu) =   -\tfrac{i}{2}(i_{Y} \omega_3 - i \, i_{Y}\omega_2)=-\tfrac{1}{2} i_Y \omega.
\end{equation*}
Here we have used that $\omega_2+i\, \omega_3$ is of type $(2,0)$ along $D_0$ and $X^{1,0}_0=Y$. 
\end{proof}
\begin{rem}\label{rem:localPhi}
	For our considerations in Section \ref{s:EnergyAsMomentMap} we observe that the previous proof works on any open, not necessarily $S^1$-invariant, subset $U\subset D$ with $H^1(U,\R)=0$. 
	 Indeed, under this assumption we always find a Hamilton function $f$ with $df=i_X\omega_1$. 
	 By replacing $\mu$ with $f$, the previous argument provides $\varphi_U\in H^0(T_Z(2)_{|U})$.
	 The condition $H^1(M,\R)=0$ then guarantees that the $\varphi_U$, which differ by a real additive constant (Corollary \ref{cor:real}), glue to a global section. 
\end{rem}
\begin{thm}\label{thm:res}
  Let $Z=Z(M)$ be the twistor space of a connected hyper-K\"ahler manifold $M$ with circle action as before with $[\omega_1/2\pi]\in H^2(M,\Z)$ and $H^1(M,\C)=0$. 
  Let $B\subset \mathbb CP^1$ be an open disk around $0\in \mathbb CP^1$ and $s\colon B\to Z$ a local holomorphic section with $s(0)=m\in M$. 
  Then
  \begin{equation}\label{eq:respsi} 
  	\mathrm{res}_\varphi(s)=\mathrm{res}_0(s^*\nabla_{\varphi})=-\tfrac{1}{2}\, i_Y\omega\left(\dot{s}(0)-\dot{s}_m(0)\right) +\mu(m), 
  \end{equation}
  up to a additive constant where $s_m$ is the twistor line through $m\in M$ and $\mu\colon M\to i\mathbb{R}$ the moment map of the circle action. 
\end{thm}

\begin{rem}
\begin{enumerate}[(i)]
\item
The right-hand side of (\ref{eq:respsi}) is well-defined because $\dot{s}(0)-\dot{s}_m(0)\in T_{F}$. 
\item 
Since the moment map $\mu\colon M\to i\R$ is only unique up to an additive constant in $i\R$, the freedom in $\varphi$ reflects the freedom in $\mu$.
\end{enumerate}
\end{rem}

\begin{proof}
We first prove the statement about the additive constant.
Let $\varphi, \varphi' \in H^0(D,T_Z^*(2))$ satisfy (\ref{eq:condphi}). 
As we have seen in Corollary \ref{cor:real}, we have $\varphi'= \varphi+ c\, \gamma$ where $c\in \R$ and $\gamma= \pi_Z^*d\lambda\otimes \pi_Z^* \tfrac{\partial}{\partial \lambda}$, cf. \eqref{eq:triv}. 
But $s^*\gamma=d\lambda \otimes \tfrac{\partial}{\partial \lambda}\in H^0(\mathbb CP^1, \mathcal{O}(-2)\otimes \mathcal{O}(2))$, so that
\begin{equation*}
\mathrm{res}_{\varphi'}(s) =s^*(\varphi + c\, \gamma)= \mathrm{res}_{\varphi}(s)+c
\end{equation*}
for any holomorphic section $s$ of $\pi_Z$. 
 
Hence it is sufficient to prove (\ref{eq:respsi}) for the section $\psi\in H^0(D,T_Z^*(2))$ of Lemma \ref{l:psi}. 
Every twistor line $s_m$ induces the same smooth (dual) splitting 
$
\theta_{s_m}\colon s_m^*T_F(2)\oplus \mathcal{O}_{\mathbb CP^1} \to T_Z(2)
$
as for the definition of $\psi$ so that we compute
\begin{equation*}
\mathrm{res}_0(s_m^*\nabla_\psi)=\mu(m). 
\end{equation*} 
Let $s\colon B\to Z$ be any local holomorphic section with $s(0)=m$. 
It defines the splitting $\theta_s\colon s^*T_F(2)\oplus \mathcal{O}_B\to s^*T_Z(2) $. 
Since $s_m(0)=s(0)$ we can compare the splittings at $\lambda=0$ and obtain 
\begin{equation}\label{eq:splittings}
\theta_{s_m}^{-1}\circ \theta_{s}(\tfrac{\partial}{\partial \lambda})=\dot{s}(0)-\dot{s}_m(0)\in T_{D_0}.
\end{equation}
Now $\psi$ is given in the dual splitting $\theta_{s_m}^*$ so that 
\begin{equation*}
\begin{aligned}
(s^*\psi)_0&=\tfrac{1}{2}\, \phi\left(\dot{s}(0)-\dot{s}_m(0)\right)+\mu(m) \\
&= -\tfrac{1}{2}\, i_Y\omega\left(\dot{s}(0)-\dot{s}_m(0)\right)+ \mu(m)
\end{aligned}
\end{equation*}
by Lemma \ref{l:psi} and (\ref{eq:splittings}). 
\end{proof}

Hence the residue $\mathrm{res}_\varphi$ is natural in several ways. 
Not only is it essentially independent of $\varphi$ or the base point (i.e. $0$ or $\infty$) but it is also the analytic continuation of the moment map $\mu\colon M\to i\mathbb{R}$ to the space of all holomorphic sections, where we identify $M$ with space of the twistor lines.

\subsection{Relation to the energy functional} \label{Subsec:RelEnergy}
We next explain the relation between the generalized energy functional $\mathcal{E}$ and the residue(s) $\mathrm{res}_\varphi$ if $M = \mathcal M_{SD}^{irr}(\Sigma,\SU(n))$ is the smooth locus of $\mathcal{M}_{SD}(\Sigma,\SU(n))$, the moduli space of solutions to the self-duality equations with structure group $\SU(n)$. Here we use the notation set up in Section \ref{ss:dh}

Recall that the circle action on $M$ is induced by the circle action on $\mathcal{C}(E)\times \Omega^{1,0}(\Sigma, \mathfrak{sl}(E))$ given by $e^{i\alpha}.(\bar{\partial}, \Phi)=(\bar{\partial}, e^{i\alpha} \Phi)$. The corresponding vector field $X$ at a point $(\bar\partial,\Phi)\in\mathcal C\times\Omega^{1,0}(\Sigma,\sln(E))$ is 
\begin{equation}\label{Eq: VFcircleaction}
X_{(\bar{\partial},\Phi)}= \frac{d}{dt}|_{t=0}(e^{it}.(\bar\partial,\Phi)) = (0, i \Phi). 
\end{equation}
Here we have used the identification $ T_{(\bar\partial,\Phi)}\left(\mathcal{C}(E)\times \Omega^{1,0}(\Sigma, \mathfrak{sl}(E))\right) = \Omega^{0,1}(\Sigma,\sln(E))\oplus \Omega^{1,0}(\Sigma,\sln(E))$. 
The moment map (with respect to $\omega_1$) is given by 
\begin{equation}
 \mu(\bar{\partial},\Phi)=- \int_{\Sigma} \mathrm{tr}(\Phi\wedge \Phi^*),
\end{equation}
as follows easily from the explicit form of the metric 
\begin{equation}
 g((\gamma,\beta),(\gamma,\beta))=2i \int_\Sigma \mathrm{tr}(\gamma^*\wedge \gamma+\beta\wedge \beta^*). 
\end{equation}
Recall moreover that the holomorphic symplectic $2$-form $\omega_{\C} = \omega_2+i\omega_3$ (with respect to $I_1$) is 
\begin{equation}\label{Eq: omegaC}
\omega_{\C}((\gamma_1, \beta_1),(\gamma_2,\beta_2))= 2i \int_\Sigma \mathrm{tr}(\beta_2\wedge \gamma_1 - \beta_1 \wedge \gamma_2),
\end{equation}
where $(\gamma_i,\beta_i)\in\Omega^{0,1}(\Sigma,\sln(E))\oplus\Omega^{1,0}(\Sigma,\sln(E)), \, i=1,2.$
If $(\gamma,\beta)\in \Omega^{0,1}(\Sigma,\sln(E))\oplus \Omega^{1,0}(\Sigma,\sln(E))$ is an arbitrary tangent vector then \eqref{Eq: VFcircleaction} and \eqref{Eq: omegaC} combine to give
\begin{equation}
i_{X_{(\bar\partial,\Phi)}}\omega_\C(\gamma,\beta) = \omega_\C (X_{(\bar\partial,\Phi)},(\gamma,\beta)) = \omega_\C((0, i \Phi),(\gamma,\beta)) =2 \int_\Sigma \mathrm{tr}(\Phi\wedge \gamma). 
\end{equation}
The holomorphic line bundle $L_Z$ on $Z$ is determined as follows (see \cite[\S 3.7]{Hitchin-hkqk} for details): 
by construction, $Z=Z_+\cup Z_-$ where $Z_+\cong \mathcal{M}_{Hod}(\Sigma, SL(n,\C))$ and $Z_-\cong \mathcal{M}_{Hod}(\overline{\Sigma}, SL(n,\C))$. 
On $Z_\pm$ there is a natural determinant line bundle $L_\pm$ with fiber $(\Lambda^{top} \ker (\mathrm{D}))^*\otimes \Lambda^{top}\,\mathrm{coker}(D)$ over the isomorphism class $[(\bar{\partial},\mathrm{D})]\in Z_\pm$ of a $\lambda$-connection $(\bar{\partial},D)$. 
Then $L_Z$ is obtained by gluing $L_+$ and $L_-^*$ over $Z_+\cap Z_-$. 

\begin{cor}\label{cor:ResEnergy}
Let $M=\mathcal{M}_{SD}^{irr}(\Sigma,\SU(n))$, $Z=Z(M)=\mathcal{M}^{s}_{DH}(\Sigma,\SL(n,\C))$ the corresponding twistor space.
Further let $\nabla_{\psi}$ be the meromorphic connection on $L_Z$ uniquely determined by $\psi\in H^0(Z,T_Z^*(2))$ as in (\ref{eq:psi}). 
If $s\colon B\to Z$ is a local holomorphic section around $0\in B$, then 
\begin{equation}
\mathcal{E}(s)=\tfrac{i}{2\pi} \, \mathrm{res}_{\psi}(s) 
\end{equation}
where $\mathcal{E}(s)$ is the energy functional of Section \ref{Sec:defene}. 
\end{cor}

\begin{proof}
First of all, let $m:=[\bar{\partial},\Phi]:=s(0)$ be the value of $s$ at $0$ and $s_m\colon B\to Z$ the corresponding twistor line through $m$. 
Further let $\widehat{s}(\lambda)=\lambda^{-1} \Phi+ \nabla^0+\lambda \Psi+\dots$ and $\widehat{s}_m(\lambda)=\lambda^{-1}\Phi+\nabla^h+\lambda \Phi^*$ be local lifts of $s$ and $s_m$ respectively. 
In particular, we have
\begin{gather}
\mathcal{E}(s)=\tfrac{1}{2\pi i} \int_{\Sigma} \mathrm{tr}(\Phi\wedge \Psi),\\
\dot{\widehat{s}}(0)-\dot{\widehat{s}}_m(0)=(\Psi-\Phi^*, \partial^{\nabla^0}-\partial^{\nabla^h}).
\end{gather}
Hence Theorem \ref{thm:res} implies 
\begin{equation}
\begin{aligned}
\mathrm{res}_\psi(s)&=-\tfrac{1}{2}\, i_Y\omega\left(\dot{s}(0)-\dot{s}_m(0)\right) +\mu(m) \\
&=- \int_\Sigma \mathrm{tr}\left(\Phi\wedge (\Psi-\Phi^*)\right)-\int_\Sigma \mathrm{tr}(\Phi\wedge \Phi^*) \\
&=- \int_\Sigma \mathrm{tr} \left(\Phi \wedge \Psi\right) \\
&=-2\pi i\, \mathcal{E}(s). 
\end{aligned}
\end{equation}
\end{proof}

\begin{remark}\label{Rem:EnergyAsResidue}
	\begin{enumerate}[(i)]
				\item We have formulated Corollary \ref{cor:ResEnergy} for the $\SL(n,\C)$-case. 
						However, Corollary \ref{cor:ResEnergy} makes sense for any complex reductive group $G_\C$ and the previous proof stills works once we replace $\mathrm{tr}$ by an appropriate non-degenerate invariant form on $\mathfrak{g}_{\C}$. 
						In the semi-simple case we take (an appropriate multiple of) the Killing form.
				\item To our best knowledge, meromorphic connections $\nabla_{\varphi}$ in terms of determinant line bundles have only been given for $M=\mathcal{M}_{SD}^{irr}(\Sigma, \C^*)$ in \cite[Theorem 5.13]{FreixasWentworth} via their theory of intersection connections on Deligne pairings.  
				Our results could be useful to extend \cite[Theorem 5.13]{FreixasWentworth} to higher rank.
	\end{enumerate}
\end{remark}

\subsection{The energy as a moment map}\label{s:EnergyAsMomentMap}
 Let $Z=Z(M)$ be the twistor space of a connected hyper-K\"ahler manifold $M$ with circle action as before with $[\omega_1/2\pi]\in H^2(M,\Z)$ and $H^1(M,\C)=0$. 
 We assume that there exists a component $N$ of real holomorphic sections of $Z\to\C P^1$ which is different from the component $M$ of twistor lines.
 We further assume that the normal bundle for any section $s\in N$ is the direct sum of $\mathcal O(1)\to\C P^1$ and that the twistor construction \cite{HKLR} yields a positive definite Riemannian metric $g^N$ induced by $\omega$.
 This implies that the evaluation map for any $\lambda\in\C P^1$ of real normal sections is a local diffeomorphism. 
 Hence, by \cite{HKLR}, $(N, g^N)$ extends to a hyper-K\"ahler manifold $(N; g^N,I^N_1, I^N_2, I^N_3)$.
 The circle action on the twistor space induces a circle action on $N$. 
Indeed, for $c\in S^1\subset \C$, and the corresponding biholomorphic map $\Phi_c\colon Z\to Z$, we define for a given section 
 $s$ the new section 
 \[
 s_c\colon \C P^1\to Z,\quad  \lambda \mapsto  \Phi_c(s(\tfrac{1}{c}\lambda)).
 \]
 Clearly, $s_c$ is real holomorphic if $s$ is real holomorphic, and because $S^1$ is connected $s$ and $s_c$ are in the same component of real holomorphic sections.
This circle action is again rotating. 

\begin{theorem}\label{the:momentum}
	Let $Z=Z(M)$ be the twistor space of a connected hyper-K\"ahler manifold $M$ with circle action as before with $[\omega_1/2\pi]\in H^2(M,\Z)$ and $H^1(M,\C)=0$. 
	Let $N\subset \mathcal{S}$ be a component of real holomorphic sections of $Z\to\C P^1$ such that the twistor construction of \cite{HKLR} yields a hyper-K\"ahler manifold $(N; g^N, I_1^N, I_2^N,I_3^N)$.
	\\
	Then $N$ has a rotating circle action, and the residue $\mathrm{res}_\psi\colon \mathcal{S}\to \C$ of the natural meromorphic connection $\nabla_\psi$ on $L_Z\to Z$ restricted to $N$ yields a moment map for the circle action with respect to $\omega_1^N$. 
	In particular, $\mathrm{res}_\varphi$ is a K\"ahler potential for $(N,g,I_2^N)$.
\end{theorem}
Note that $H^1(N,\R)$ might not be zero so that general arguments do not even guarantee the existence of a moment map on $N$. 
\begin{proof}
	For every $s\in N$, there exist open neighborhoods $U\subset N$, $V\subset M=Z_0$ of $s$ and $s(0)$ respectively such that there is a biholomorphism
	\begin{equation*}
	\Psi\colon Z(U)\to Z(V)
	\end{equation*}
	of the twistor spaces of $U$ and $V$.
	It is compatible with the fibrations to $\C P^1,$ the real structures and the twisted relative symplectic forms.
	Even though $U$ might not be $S^1$-invariant, there is a holomorphic line bundle $L_{Z(U)}$ -- induced by a hyperholomorphic line bundle $L_U$ over $U$ -- with a meromorphic connection $\nabla_{\varphi_U}$ as before, cf. Remark \ref{rem:iXomega1sufficient} and \ref{rem:localPhi}.
	Theorem \ref{thm:res} implies that $\mathrm{res}_{\varphi_U}\colon U\to i\R$, $s\mapsto \mathrm{res}_0(s^*\nabla_{\varphi_U})$, satisfies $d\,\mathrm{res}_{\varphi_N}=i_{X_N}\omega_1$ where $X_N$ is the vector field of the circle action on $N$. 
	\\
	On the other hand, if $d=\mathrm{deg}(s^*L_Z)$ for any $s\in N$, then 
	\begin{equation}\label{eq:IsomorphismBetweenLZLZN}
	\Psi^*L_Z\otimes \pi_{Z(N)}^*\mathcal{O}(-d)\cong L_{Z(U)}
	\end{equation}
	over $Z(U)$. 
	By Proposition \ref{p:mero}, the meromorphic connections $\nabla_{\varphi_N}$ and $\nabla_\varphi$ (tensored with an appropriate meromorphic connection $\nabla^d$ on $\mathcal{O}(-d)$ if necessary) on $L_{Z(U)}$ and $L_Z$ respectively are exchanged under \eqref{eq:IsomorphismBetweenLZLZN} if the additive constants in the residues $\varphi_U$ and $\psi$ are appropriately chosen. 
	Hence $\mathrm{res}_\psi(s)=\mathrm{res}_{\varphi_U}(s)$ for any $s\in U$ up to an additive constant.  
	Since $\mathrm{res}_\psi$ is globally defined on $N$, it is a moment map for the circle action on $N$. 
	
	The second part is now standard, see \cite[\S 3 (E)]{HKLR}. 
\end{proof}

\section{The energy and the Willmore functional}\label{Sec:Willener}

We have seen in Theorem \ref{harmonicenergy} that for twistor lines the energy $\mathcal E$ is directly related to the harmonic map energy of the corresponding equivariant harmonic map.
In \cite{HH2}, non-admissible $\tau$-negative real holomorphic sections of the rank 2 Deligne--Hitchin moduli spaces have been constructed. These sections
correspond to equivariant Willmore surfaces, for definitions see  Section \ref{sec:Willhigher} below. We will exhibit an explicit formula relating the Willmore energy of the surface with the energy of the corresponding section of $\mathcal M_{DH}\to\C P^1$. 
Before we can state the main results, we need an auxiliary tool: the dual surface construction. In the following sections we restrict to rank 2  Deligne--Hitchin moduli spaces.

\subsection{The dual surface construction}

Consider a holomorphic section $s$ of the Deligne--Hitchin moduli space. We assume that $s(0)$ is a stable Higgs pair with nilpotent Higgs field.
The section $s$
admits  a (local) lift $\nabla^\lambda=\lambda^{-1}\Phi+\nabla+\lambda\Psi +\dots$ such that $\Phi$ is nilpotent.

 By assumption, the kernel bundle $L$ of $\Phi$
has negative degree. Choose a complementary subbundle $L$ and apply the gauge transformation $h(\lambda) = \mathrm{diag}(\lambda^{-1},1)$, written with respect to $L\oplus L^*$ 
to $\nabla^\lambda$, cf. the proof of Proposition \ref{twistablepro}. In this way, we obtain a new $\C^*$-family of flat $\SL(2,\C)$-connections 
\begin{equation}\label{parallelsurface}
\hat\nabla^\lambda=\nabla^\lambda.h(\lambda).
\end{equation}
With respect to $L\oplus L^*$ we may write
\[
\Phi = \left(\begin{array}{cc} 0 & \phi\\ 0 &0\end{array}\right),\qquad \nabla = \left(\begin{array}{cc} \nabla^L & \alpha \\ \beta & \nabla^{L^\perp}\end{array}\right),\qquad \Psi = \left( \begin{array}{cc} \psi_{11} & \psi_{12} \\ \psi_{21} & \psi_{22}\end{array}\right),
\]
with $\phi\in H^0(KL^{-2}), \, \beta\in\Omega^{1,0}(L^2)$. By a computation analogous to the one in the proof of Proposition \ref{twistablepro} we see that
\[
\hat\nabla^\lambda = \nabla^\lambda.h(\lambda) = \Phi + \nabla.h(\lambda) + \lambda\Psi.h(\lambda) + \dots =  \lambda^{-1}\left(\begin{array}{cc} 0 & 0\\ \beta &0\end{array}\right) + \left(\begin{array}{cc} \nabla^L & \phi \\ \psi_{21} & \nabla^{L^\perp}\end{array}\right) +\lambda\left(\begin{array}{cc} \psi_{11} & \alpha \\ * & \psi_{22}\end{array}\right) + \dots
\]
Note that although the corresponding family of $\lambda$-connections has a limit as $\lambda\to 0$,  this family is not the lift of any holomorphic section $\hat s\colon \C\to \mathcal M_{DH}$, as the Higgs pair $(\overline{\partial}^{\hat\nabla},\hat\Phi)$ at $\lambda=0$ is unstable: Indeed, we have
\[
\overline{\partial}^{\hat\nabla} = \left(\begin{array}{cc} \overline{\partial}^L & 0\\ * & \overline{\partial}^{L^\perp}\end{array}\right),\qquad \hat\Phi = \left(\begin{array}{cc} 0 & 0 \\ \beta & 0\end{array}\right).
\]
Thus, the holomorphic subbundle $L^*$ is the kernel bundle of $\hat\Phi$ and has positive degree.
Still, we can interpret $\lambda\mapsto\hat\nabla^\lambda$ as a map $\hat s$ into the space of holomorphic $\lambda$-connections,
and consider its energy $E(\hat s)$ as defined in \eqref{e1}. This is well-defined, and invariant under holomorphic families of gauge transformations $\lambda\mapsto g(\lambda)$ which extend holomorphically to $\lambda=0$ (see the proof of Proposition \ref{defene}).

With 
\[
\hat\Psi = \left(\begin{array}{cc} \psi_{11} & \alpha \\ * & \psi_{22}\end{array}\right)
\]
a computation analogous to the proof of Proposition \ref{twisttheenergy} yields the following formula relating the energy of $\hat s$ to that of $s$.
Note that on the right hand side of the formula $\mathcal E(s)$ appears with a factor 1 as opposed to the formula in Proposition \ref{twisttheenergy}.
\begin{proposition}\label{twisttheenergy2}
Let $s\colon \mathbb CP^1\to\mathcal M_{DH}$ be an irreducible holomorphic section such that the stable Higgs pair $s(0) = (\overline{\partial},\Phi)$  on $\Sigma$ has nilpotent Higgs field. Then we have 
\[
\mathcal E(\hat s) = \mathcal E(s) - \deg L,
\]
where $L$ is the kernel bundle of $\Phi$ and $\hat s$ is determined by \eqref{parallelsurface}.
\end{proposition}
\begin{remark}
Consider an equivariant minimal surface $f\colon \tilde\Sigma\to S^3=\SU(2)$ and the associated family of flat connections 
\[
\nabla^\lambda = \lambda^{-1}\Phi + \nabla - \lambda\Phi^*,
\]
where $(\nabla,\Phi)$ is a solution of \eqref{eq:hmS3}. 
The Higgs field $\Phi$ is nilpotent as the surface is given by a conformal harmonic map, and we can apply the construction \eqref{parallelsurface}. Denote the kernel bundle of $\Phi$ by $L$ and write with respect to $E = L\oplus L^\perp$ 
\[
\nabla^\lambda = \lambda^{-1}\Phi + \nabla - \lambda\Phi^* = \begin{pmatrix}
\nabla^L & \alpha + \lambda^{-1}\phi \\ -\alpha^* - \lambda\phi^* & \nabla^{L^\perp}
\end{pmatrix}.
\]
The dual surface construction then yields the family
\[
\hat\nabla^\lambda = \begin{pmatrix}
\nabla^L & \lambda\alpha + \phi \\ -\lambda^{-1}\alpha^* - \phi^* & \nabla^{L^\perp}
\end{pmatrix},
\]
which satisfies the same reality condition as 
$\nabla^\lambda$, i.e. both are unitary for $\lambda\in S^1$. Moreover $\hat\nabla^\lambda$ has nilpotent Higgs field as well. It therefore gives another conformal harmonic map into $S^3=\SU(2)$ which is branched at the zeros of the Hopf differential of the surface $f$. This construction is well-known in classical surface
theory, and is sometimes called the parallel or dual surface of the initial minimal surface $f$, see \cite{Koba} and the references therein.
\end{remark}

\begin{remark}\label{rem:dualsurfacereality}
Suppose $(\nabla,\Phi)$ is a solution to the self-duality equations with nilpotent Higgs field $\Phi$ corresponding to an equivariant minimal surface $f\colon \tilde\Sigma\to H^3$. Let $L$ again be the kernel bundle of $\Phi$. The associated $\tau$-real family of flat connections on $E = L\oplus L^\perp$ is of the form
\[
\nabla^\lambda = \lambda^{-1}\Phi + \nabla + \lambda\Phi^* = \begin{pmatrix}
\nabla^L & \alpha + \lambda^{-1}\phi \\ -\alpha^* + \lambda\phi^* & \nabla^{L^\perp}
\end{pmatrix}.
\]
The dual surface construction yields the family
\begin{equation}\label{dusufa}
\hat\nabla^\lambda = \begin{pmatrix}
\nabla^L & \lambda\alpha + \phi \\ -\lambda^{-1}\alpha^* + \phi^* & \nabla^{L^\perp}
\end{pmatrix}.
\end{equation}
We observe that 
\[
\overline{\hat\nabla^{\lambda}}^* = \begin{pmatrix}
\overline{\nabla^L}^* & \bar\lambda^{-1}\alpha - \phi \\ -\bar\lambda\alpha^* - \phi^* & \overline{\nabla^{L^\perp}}^*
\end{pmatrix} 
= \begin{pmatrix}
\overline{\nabla^L}^* & -(-\bar\lambda^{-1}\alpha + \phi) \\ -(\bar\lambda\alpha^* + \phi^*) & \overline{\nabla^{L^\perp}}^*
\end{pmatrix}.
\]
Thus, $\hat\nabla^\lambda$ satisfies a different reality condition than $\nabla^\lambda$. In fact, it follows that the family $\hat\nabla^\lambda$ does not give an equivariant harmonic map to $H^3$ but  an equivariant harmonic map  $\tilde \Sigma \to dS^3 = \SL(2,\C)/\SU(1,1)$ into the de Sitter space, see \cite[Section 3]{BHR}.
Because the Higgs field
\[\begin{pmatrix} 0 &0 \\\alpha^*&0\end{pmatrix}\] of the family  $\hat\nabla^\lambda$ is also nilpotent
the corresponding equivariant harmonic map into de Sitter space is conformal as well. 

The de Sitter space $dS^3$ has the identification as the space of oriented  circles on a fixed $2$-sphere. We consider the 2-sphere as the equatorial $2$-sphere $S_\infty$ in the $3$-sphere which separates two hyperbolic $3$-balls.
 The space of oriented circles $C$ in the 2-sphere can be identified with the space of oriented $2$-spheres $S$ in the $3$-sphere which intersect $S_\infty$ perpendicularly, i.e. $C=S\cap S_\infty$ as oriented submanifolds of $S^3$. 
 In this interpretation, the equivariant conformal harmonic map into de Sitter 3-space yields a map into the space of oriented 2-spheres in the 3-sphere. We will see in Section \ref{sec:light} below, that the latter map 
  is the mean curvature sphere of the minimal surface $f$ in $H^3\subset S^3,$ i.e., the map which associates to a point $p$ of the surface the best touching 2-sphere
  of $f$ at $p.$
\end{remark}

Let $s$ be a twistor line given by a nilpotent Higgs pair $s(0),$ and apply the dual surface construction.
From \eqref{dusufa} we can directly compute that $E(\hat s)\geq0$ with equality if and only if $\nabla$ is reducible, i.e. $\alpha =0$.
As an application of Proposition \ref{twisttheenergy2}
we reobtain the well-known energy estimate 
\[
0> \mathcal E(s)\geq \deg L\geq 1-g,
\]
where $g$ is the genus of the surface $\Sigma$.

\subsection{The Willmore functional and the energy of higher sections}\label{sec:Willhigher}
 
 A solution $(\nabla,\Phi)$ of the self-duality equations with nilpotent Higgs field $\Phi\neq0$ gives rise to a branched conformal harmonic map, i.e., an equivariant minimal surface $f\colon \tilde\Sigma\to H^3$ with branch points. The basic invariant of the equivariant minimal surface is the area of a fundamental piece, which is determined by the energy of the harmonic map, i.e.,
\[Area(f)=energy(f)=-4\pi\mathcal E(s),\]
where $s$ is the $\tau$-real holomorphic section of the Deligne--Hitchin moduli space corresponding to the solution $(\nabla,\Phi)$ of the self-duality equations.

The Willmore energy of a conformal immersion $f\colon\tilde\Sigma\to M$ into a Riemannian $3$-manifold $M$ is given by
\[
\mathcal W(f)=\int_\Sigma \left(H^2-K+\bar K \right)dA,
\]
where $dA$ is the induced area form, $K$ is the curvature of the induced metric, $H=\tfrac{1}{2}\text{tr}(II)$ is the mean curvature, i.e., the half-trace of the second fundamental form $II$, and for $p\in \tilde\Sigma$ the quantity $\bar K_p$ is the sectional curvature of the tangent plane 
$T_{f(p)}f(\tilde\Sigma)\subset T_{f(p)}M$. 
It was known already to Blaschke that the Willmore functional for surfaces in $\R^3$ or $S^3$ is invariant under M\"obius transformations of the target space.
It was first shown in \cite{chen} that the Willmore integrand is actually invariant under conformal changes of the metric on $M$.

In the case of an equivariant,  immersed minimal surface $f\colon\tilde\Sigma\to H^3$ into hyperbolic 3-space, $H=0$ and the Willmore functional therefore equals to
\[
\mathcal W(f)=-\int_\Sigma KdA-\int_\Sigma dA=2\pi(2g-2)+4\pi\mathcal E(s).
\]
We apply the dual surface construction  \eqref{parallelsurface}
 to the corresponding holomorphic section $s$ with nilpotent Higgs field $\Phi$.  We obtain a new family of $\lambda$-connections $\hat s$. Let $L$ be the kernel bundle of  $\Phi$. Because $\Phi$ can be interpreted as the $(1,0)$-part of the differential of the minimal surface,
  the zeroes of the Higgs field are branch points of $f$. Since we assume that $f$ is not branched, we must have $\mathrm{deg}(L) =1-g$. Hence, Proposition \ref{twisttheenergy2} implies
\begin{equation}\label{eq:willmoreenergy}
\mathcal W(f)=4\pi\, E(\hat s).
\end{equation}

The conformally invariant Willmore integrand $\left(H^2-K+\bar K \right)dA$ can be generalized to a class of branched conformal maps into the conformal 3-sphere. The extra assumption is that the mean curvature sphere (the exact definition is given in Section \ref{sec:light} below) extends through the branch points of the conformal map,  see \cite{BFLPP} and related literature. We will see in Section \ref{sec:light}  that this assumption holds for branched minimal surfaces in hyperbolic 3-space, yielding \eqref{eq:willmoreenergy} in this more general situation. In fact, we obtain an equality for the Willmore integrand:
\begin{proposition}\label{prowiin}
Let $f\colon \Sigma\to H^3$ be an equivariant branched minimal surface with
associated $\C^*$-family of flat connections $\nabla^\lambda$. Then the conformally invariant Willmore integrand is given by
\[
-2i\,\mathrm{tr}(\hat\Phi\wedge\hat\Psi),
\]
where $\hat\nabla^\lambda:=\nabla^\lambda.h(\lambda)=\hat\nabla+\lambda^{-1}\hat\Phi+\lambda\hat \Psi$ is given by the dual surface construction. 
In particular 
\[
\mathcal W(f)=4\pi E(\hat s),
\]
where $\hat s$ is the family of $\lambda$-connections determined by $\hat\nabla^\lambda$.
\end{proposition}
A proof  is given in Section \ref{sec:light} below using notions from conformal surface geometry.

In \cite{HH2} it was shown that there exist compact Riemann surfaces $\Sigma$ whose associated Deligne--Hitchin moduli spaces admit $\tau$-negative holomorphic sections $s$ with the following properties:
\begin{enumerate}
\item the Higgs field $\Phi$ is nilpotent, where $s(0)=[\bar\partial,\Phi]$;
\item the section $s$ is not admissible: for a lift $\nabla^\lambda$ with $\overline{\nabla^{\bar\lambda^{-1}}}=\nabla^\lambda.g(\lambda)$
the Birkhoff factorization $g=g_+g_-$  fails along a real analytic (not necessarily connected)  curve $\gamma\subset \Sigma$ (see Remark \ref{rem:lift-section-admissible});
\item On $\Sigma\setminus\gamma$ the section $s$ gives rise to an (equivariant) conformal harmonic map which extends through the boundary 2-sphere at infinity of the hyperbolic 3-space, yielding a M\"obius equivariant Willmore surface $f\colon \tilde\Sigma\to S^3=H^3\cup S^2\cup H^3$.
\end{enumerate}
It is a natural guess that the energy $\mathcal E(s)$ is related to the Willmore energy of a fundamental piece of $f$.
We remark that \cite[Theorem 9]{BabBob} can be interpreted as this relation in the case of $\Sigma$ being of genus 1. Our main result here is the following theorem, whose proof we postpone to section \ref{sec:proofs}.

\begin{theorem}\label{willmoreenergy}
Let $\Sigma$ be a compact Riemann surface. 
Let $s$ be a $\tau$-negative holomorphic section of the Deligne--Hitchin moduli space $\mathcal M_{DH}(\Sigma,\SL(2,\C)$ satisfying the above mentioned properties (1)-(3). Let $L$ be the kernel bundle of the nilpotent Higgs field of the Higgs pair $s(0) = [\bar\partial,\Phi]$.
Then the energy of $s$ is related to the Willmore energy of a fundamental piece of the surface $f$ by the formula
\[
\mathcal E(s)=\tfrac{1}{4\pi}\mathcal W(f)+\mathrm{deg}(L).
\]
\end{theorem}

\begin{remark}
This geometric interpretation might have applications in theoretical and mathematical physics, in particular in the AdS/CFT-correspondence (see for example \cite{Maldacena_1998,Alday_Maldacena_2009} for more details):
Consider a minimal surface in a totally geodesic $H^3\subset \mathrm{AdS}^4$ which intersects the boundary at infinity.
If the surface extends to a Willmore surface in $S^3$, giving rise to a  $\tau$-negative holomorphic section of the Deligne--Hitchin moduli space of a compact Riemann surface, the finite part of the area functional is given by the Willmore energy of the surface. If we additionally have a symmetry between the two pieces of the minimal surface in the two components of $H^3$
in $S^3=H^3\cup S^2\cup H^3$, the Willmore energy is given in terms of the energy of the section. A similar relation holds for space-like minimal surfaces in $\mathrm{AdS}^3$.
\end{remark}

\subsection{The lightcone approach to conformal surface geometry}\label{sec:light}
Our proofs of Proposition \ref{prowiin} and Theorem \ref{willmoreenergy} will use some concepts of conformal surface geometry in the lightcone model, which we recall here. We refer to \cite{BuPP,BuCa,Qui} for details.

Consider $\R^{4,1}$ with the standard Minkowski inner product
\[ 
\langle.,.\rangle=-(dx_0)^2+(dx_1)^2+\ldots +(dx_4)^2
\]
and the lightcone
\[
\mathcal L=\{x\in\R^{4,1}\mid \langle x,x\rangle=0\}.
\]
Then the map $\R^4\to P\R^{4,1}, (x_1,x_2,x_3,x_4)\mapsto [1:x_1:x_2:x_3:x_4]$ restricts to a natural diffeomorphism between the $3$-sphere $S^3$ and the projectivization $P\mathcal L\subset P\R^{4,1}$. There   exists a natural conformal structure on $P\mathcal L$ induced by $\langle.,.\rangle$, which contains the round metric on $S^3$.
If $\sigma$ is a (local) section of $\pi\colon \mathcal L\to P\mathcal L$ then the conformal structure is represented by the Riemannian metric $g_\sigma$ defined as
\[
g_\sigma(X,Y):=\langle d\sigma(X),d\sigma(Y)\rangle.
\]
The round metric is obtained from the lift $\sigma([x])= \frac{x}{x_0}, [x]\in P\mathcal L$. 
The space of orientation preserving conformal diffeomorphisms of $S^3\cong P\mathcal L$ is then given by $PSO(4,1)$ (via its natural action on $P\R^{4,1}$). 
Those transformations are also called M\"obius transformations.
We will often consider the conformal 3-sphere as the union $S^3 = H^3\cup S^2\cup H^3$, i.e. as the union of two hyperbolic balls separated by an equatorial $S^2$. In the lightcone model  a 2-sphere $S^2$ can be written as $P(v^\perp \cap \mathcal L)$ for a space-like vector $v\in\R^{4,1}$. It is known that the complement $\{[x]\in P\mathcal L\colon \langle x,v\rangle\neq 0\}$ is conformal to $H^3\cup H^3$. In particular, we note that a 2-sphere $S^2\subset S^3$ corresponds to a subspace of $\R^{4,1}$ of signature $(3,1)$.

Consider a conformal immersion $f\colon \Sigma\to P\mathcal L$ from a Riemann surface. There exists a real rank $4$ vector bundle $\mathcal S\subset \underline{\R^{4,1}}$ locally defined with respect to a holomorphic coordinate $z$ and a local lift $\hat f$ of $f$ to $\R^{4,1}$ as
\[
\mathcal S\otimes\C=\text{span}(\hat f,\hat f_z,\hat f_{\bar z},\hat f_{z,\bar z}),
\]
where for a function $g$ we denote $g_z:=\tfrac{\partial g}{\partial z}$ and so on. The real rank 4 bundle is well-defined, and $\langle.,.\rangle$ restricts to an inner product of type $(3,1)$. Under the correspondence between $2$-spheres in $S^3$ and subspaces of signature $(3,1)$ in $\R^{4,1}$ the bundle $\mathcal S$ can be interpreted as a family of $2$-spheres. It is  called the mean curvature sphere congruence associated with $f$. 
Its orthogonal complement is denoted by $\mathcal N$ and we obtain an induced decomposition of the trivial
connection on $\underline{\R^{4,1}}=\mathcal S\oplus\mathcal N$ into diagonal and off-diagonal parts
\[
d=\mathcal D_\mathcal S+\beta,
\]
where $\beta$ is tensorial and $\mathcal D_\mathcal S$ is a connection. The Willmore energy of the surface is then given by
\[
W(f)=-\tfrac{1}{4}\int_\Sigma\text{tr}(*\beta\wedge \beta),
\]
where $*dz=i dz,*d\bar z=-id\bar z$. A surface is Willmore if and only if
\[
d^{\mathcal D_\mathcal S}*\beta=0,
\]
which is equivalent to the flatness of the family of \[SO(4,1)_\C=SO(5,\C)\]connections
\[
\lambda\in\C^*\mapsto \mathcal D^\lambda=\mathcal D_\mathcal S+\lambda^{-1}\beta^{(1,0)}+\lambda \beta^{(0,1)}.
\]
The equivariant Willmore surfaces constructed in \cite{HH2} have the additional property that they are minimal in two hyperbolic balls separated by the boundary at infinity $S^2\subset S^3=P\mathcal L$. This condition is equivalent to
the fact that there exists a space-like vector $v$ of length 1 which is contained in $\mathcal S_p$ for all $p\in\Sigma$, see \cite{BuPP,BuCa,Qui} for a proof. Therefore $v$ is also parallel with respect to $\mathcal D^\lambda$ for all $\lambda\in\C^*$. After applying a M\"obius transformation we can assume that $v=e_4$.

In order to compare the $\SL(2,\C)$-family $\nabla^\lambda$ with the $\SO(5,\C)$-family $\mathcal D^\lambda$ of flat connections coming from the Willmore surface, we make use of the following
model of $\R^{4,1}$: Consider the space \[\mathcal H:=\{A\in\mathfrak{gl}(2,\C)\mid \bar A^T=A\}\]
of hermitian 2 by 2  matrices, and \[V=\mathcal H\oplus\R\]
equipped with the Lorentzian inner product
given by the quadratic form
\[q(A,r):=-\det(A)+r^2.\]
An isometry $\Psi\colon \R^{4,1}\to V$ is given by
\[
\Psi(x_0,x_1,x_2,x_3,x_4)=\left(\begin{pmatrix} x_0+x_1 & x_2+i x_3\\x_2-ix_3& x_0-x_1\end{pmatrix},x_4\right).
\]
Let $\Sigma$ be a Riemann surface.
Consider a $\C^*$-family $\nabla^\lambda$ of flat $\SL(2,\C)$-connections of the self-duality form (on the trivial $\C^2$ bundle over $\Sigma$ with standard hermitian metric) corresponding to an equivariant minimal surface $f:\tilde \Sigma\to H^3$ on the universal covering, i.e.
\[
\nabla^\lambda=d+\xi^\lambda=d+\lambda^{-1}\xi_{-1}+\xi_0+\lambda\xi_1.
\]
We have that  $\xi_{-1}\in\Omega^{(1,0)}(\Sigma,\mathfrak{sl}(2,\C))$ is nilpotent,
$\xi_{1}\in\Omega^{(0,1)}(\Sigma,\mathfrak{sl}(2,\C))$, $\xi_{0}\in\Omega^1(\Sigma,\mathfrak{sl}(2,\C))$
with
\[
d\xi^\lambda+\xi^\lambda\wedge\xi^\lambda=0
\]
for all $\lambda\in\C^*$. The family $\xi^\lambda$ satisfies the following reality condition
\[
\overline{\xi_{-1}}^T=\xi_1\quad \text{ and } \quad \overline{\xi_0}^T=-\xi_0,
\]
which implies 
\begin{equation}\label{eq:omegareal}
-\xi^{-\bar\lambda^{-1}} = \overline{\xi^\lambda}^T. 
\end{equation}
Consider a parallel frame $F\colon\C^*\times\tilde \Sigma\to\SL(2,\C)$ for $\nabla^\lambda$, i.e.,
\begin{equation}\label{eq:Flambdaparallel}
dF^\lambda=-\xi^\lambda F^\lambda
\end{equation}
with $F^\lambda(p)=\text{Id}$ for some fixed $p\in\tilde \Sigma$.
By \eqref{eq:omegareal} we have 
\begin{equation}\label{eq:Flambdareal}
d\overline{F^\lambda}^T = -\overline{F^\lambda}^T\overline{\xi^\lambda}^T = \overline{F^\lambda}^T\xi^{-\bar\lambda^{-1}}.
\end{equation}
On the other hand, \eqref{eq:Flambdaparallel} implies 
$d((F^\lambda)^{-1}) = (F^{\lambda})^{-1}\xi^\lambda$,
so that 
\begin{equation}\label{eq:Flambdarealinverse}
(F^\lambda)^{-1} = \overline{F^{-\bar\lambda^{-1}}}^T.
\end{equation}
 The corresponding equivariant minimal surface $f$ in $H^3$ is now given by
\[
\hat f=((F^{-1})^{-1}F^1,1)=(\overline F^T F,1)\colon \tilde M\to  V,
\]
where we put $F:=F^1$ for short. As $\det(F)=1$, $f=\R \hat f$ maps to the projectivized lightcone
$P\mathcal L\subset PV$. Let $z$ be a local holomorphic coordinate on $U\subset \Sigma$. The complexified mean curvature sphere $\mathcal S\otimes \C$ is spanned over $U$ by
\[
\hat f, \quad \tfrac{\partial}{\partial z}\hat f=2(\overline F^T\xi_{-1}(\tfrac{\partial}{\partial z}) F,0),\quad \tfrac{\partial}{\partial \bar z}\hat f=2(\overline F^T\xi_{1}(\tfrac{\partial}{\partial \bar z}) F,0),\quad  \tfrac{\partial}{\partial z}\tfrac{\partial}{\partial \bar z}\hat f=\mu(\overline F^T F,0),
\]
where the local function $\mu$ is defined by
\[
\mu\text{Id}=2 (\xi_{-1}\wedge\xi_1-\xi_1\wedge \xi_{-1})(\tfrac{\partial}{\partial z},\tfrac{\partial}{\partial \bar z}),
\]
as a short computation using flatness of $d+\xi^\lambda$ shows.
From this we obtain the following frame of $\mathcal S\otimes \C$:
\begin{equation}\label{inducedframe}
\hat f, \quad (\overline F^T\xi_{-1}(\tfrac{\partial}{\partial z}) F,0),\quad (\overline F^T\xi_{1}(\tfrac{\partial}{\partial \bar z}) F,0), \quad (0,1).\end{equation}
In particular, $(0,1)\in V$ is a constant space-like vector contained in $\mathcal S_p$ for all $p\in \tilde \Sigma$.
Moreover, by the Riemann extension theorem, we
observe that the mean curvature sphere extends through the branch points of $f$ (given by the zeros of $\xi_{-1}$).

Note that \eqref{inducedframe}  yields an induced frame of the flat rank 5 bundle $V$ by extending the mean curvature sphere bundle by a constant length $1$ section  of its orthogonal complement. We want to describe the connection $\mathcal D^\lambda$ with respect to this frame.

Locally, on open sets where $F$ is well-defined and where we have a holomorphic coordinate $z$, we can find an $\SU(2)$-frame such that
\begin{equation}\label{eq:omega-pm}
\begin{split}
\xi_{-1}&=\begin{pmatrix}0& e^udz\\0&0\end{pmatrix}\\
\xi_0&=\begin{pmatrix}\tfrac{1}{2} u_z dz-\tfrac{1}{2} u_{\bar z} d\bar z& -e^{-u} \bar q d\bar z\\e^{-u}q dz & -\tfrac{1}{2} u_z dz+\tfrac{1}{2} u_{\bar z} d\bar z\end{pmatrix}\\
\xi_1&=\overline{\xi_{-1}}^T=\begin{pmatrix}0& 0\\ e^ud\bar z & 0\end{pmatrix}.\\
\end{split}
\end{equation}
The locally defined function $u$ is determined by the induced metric $g$ (from the hyperbolic minimal surface)  by
\[g=e^{2u} dz\otimes d\bar z,\] and $u_z,u_{\bar z}$ are determined by
$u_z dz+ u_{\bar z} d\bar z=du$
and $q$ is a holomorphic function (representing the Hopf differential $q(dz)^2$ of the surface).

We obtain the following (moving) frame for $V\otimes \C$
\begin{equation}
\begin{split}
\psi_1&:=\left(\bar F^T F,0\right),\quad \psi_2:=\left(\bar F^T\begin{pmatrix} 0&1\\0&0\end{pmatrix} F,0\right),\quad \psi_3:=\left(\bar F^T\begin{pmatrix} 0&0\\1&0\end{pmatrix} F,0\right),\\
\psi_4&:=\left(0,1\right),\quad \psi_5:=\left(\bar F^T \begin{pmatrix} 1&0\\0&-1\end{pmatrix} F,0\right).
\end{split}
\end{equation}
Due to \eqref{inducedframe} and \eqref{eq:omega-pm} the mean curvature sphere bundle
$\mathcal S\otimes \C\subset V\otimes \C$ is spanned by $\psi_1,\dots,\psi_4$. 

\begin{lemma}\label{con1SN}
With respect to the frame $(\psi_1,\dots,\psi_5)$, the connection $\mathcal D^\lambda=\mathcal D+\lambda^{-1} \beta^{1,0}+\lambda\beta^{0,1}$ is given by
\begin{equation}\label{eq:con1SN}
d+\begin{pmatrix} 
0 & 0 &-e^u & 0 & 0\\
-2e^u & u_z & 0 & 0 & 0\\
0 & 0 &-u_z &0&  \lambda^{-1}2q e^{-u}\\
0 & 0 &0 &0&0\\
0 & -\lambda^{-1}q e^{-u} &0 &0&0\\
\end{pmatrix}dz +\begin{pmatrix} 
0 & -e^u &0 &0&0\\
0 & -u_{\bar z} & 0&0&\lambda 2\bar q e^{-u}\\
-2e^u & 0 &u_{\bar z} &0&0\\
0 & 0 &0 &0&0\\
0 & 0 &-\lambda \bar q e^{-u} &0&0\\
\end{pmatrix} d\bar z.
\end{equation}
\end{lemma}
\begin{proof}
Let $\xi_{\pm} := \xi^{\pm 1} = \pm \xi_{-1} + \xi_0 \pm \xi_{1}$. Then, by \eqref{eq:Flambdareal} we have 
$dF = -\xi_+F$ and $d\bar F^T  = \bar F^T\xi_-$. It follows that if $U\in\mathfrak{gl}(2,\C)$ is a constant matrix, then 
\[
d(\bar F^TUF) = \bar F^T(\xi_-U-U\xi_+)F.
\]
Since $\psi_1,\psi_2,\psi_3,\psi_5$ are of the form $(\bar F^TUF,0)$ and obviously $d\psi_4=0$ the lemma follows by direct calculation using the explicit form of $\xi_{\pm}$ provided by \eqref{eq:omega-pm}.
\end{proof}

\begin{remark}\label{rem:hatF}
Consider the trivial bundle $\underline{V\otimes\C} = \underline{\mathfrak{gl}(2,\C)}\oplus \underline{\C}$. Together with the above quadratic form this becomes an $\mathrm{SO}(5,\C)$-bundle. Note that we have an action by $\SL(2,\C)$-valued functions via 
$\hat FU = \bar F^T U F$. It thus follows that $\hat F$ gives a gauge transformation from the standard frame given by 
\[
e_1 = \left(\begin{pmatrix} 1&0\\0&1\end{pmatrix},0\right),\;\; e_2=\left(\begin{pmatrix} 0&1\\0&0\end{pmatrix},0\right),\;\; e_3=\left(\begin{pmatrix} 0&0\\1&0\end{pmatrix},0\right),\;\; e_4=\left(0,1\right),\;\; e_5 =\left( \begin{pmatrix} 1&0\\0&-1\end{pmatrix},0\right)
\]
to the frame $\psi_1,\dots,\psi_5$. Since $\psi_4 = (0,1)$ is parallel, $\mathcal D+\beta$ induces a connection on the subbundle $\underline{\mathfrak{gl}(2,\C)}$.  Therefore, Lemma \ref{con1SN} implies that 
\[
\mathcal D+\beta = d + \hat F^{-1}d\hat F.
\]
\end{remark}

\begin{lemma}\label{lem:Dlambdadualsurface}
Let $\nabla^\lambda = d+\xi^\lambda$ be as above. 
Consider the trivial rank $5$ bundle $V\otimes \C = \underline{\mathfrak{gl}(2,\C)}\oplus \underline{\C}$ equipped with the connection $\hat D^\lambda$ defined by \[
\hat D^\lambda (A,f) = (dA + \hat\xi^{(-\lambda)}A - A\hat\xi^\lambda,df),
\]
where $\hat\xi$ is obtained from $\xi$ by the dual surface construction, i.e. $\hat\nabla^\lambda = \nabla^\lambda.\tilde h(\lambda)$,
where $\tilde h(\lambda) = \mathrm{diag}(1,\lambda)$. Then with respect to the frame
\[
\tilde e_1 = e_5,\quad \tilde e_2 = e_2,\quad \tilde e_3 = -e_3,\quad \tilde e_4 = e_4,\quad \tilde e_5 = e_1
\]
the connection $\hat D^\lambda$ is given by equation \eqref{eq:con1SN}.
Thus, $\hat D^\lambda$ and $\mathcal D^\lambda$ are gauge equivalent by a $\lambda$-independent gauge transformation. 
\end{lemma}
\begin{proof}
Obviously $\hat D^\lambda \tilde e_4 = 0$ and for $i=1,2,3,5$ we have $\tilde e_i = (E_i,0)$ with a constant matrix $E_i\in\mathfrak{gl}(2,\C)$. Thus, 
\[
\hat D^\lambda \tilde e_i = (\hat\xi^{(-\lambda)}E_i-E_i\xi^\lambda,0).
\]
A direct calculation then yields the connection matrix. 

Note that in the notation of Remark \ref{rem:hatF} we have that $\tilde e_i $ is obtained from $e_i$ by multiplying the matrix part by $\mathrm{diag}(1,-1)$ and leaving the scalar part unchanged. Denote this map $e_i\mapsto \tilde e_i$ by $S$. Then $S^{-1} = S$ and we have with the notation of Lemma \ref{con1SN}
\[
\psi_i = \hat FS\tilde e_i.
\]
Clearly, the gauge transformation $\hat G = \hat F\circ S$ is independent of $\lambda$ and satisfies 
$\hat G^{-1}\circ \mathcal D^\lambda \circ \hat G = \hat D^\lambda$
as can be checked in the frame $\{\tilde e_i\}$.
\end{proof}

\subsection{Proofs}\label{sec:proofs}
We will now use the theory of the previous section to give the proofs of the results formulated in section \ref{sec:Willhigher}. 
\begin{remark}\label{energy-integrand-constant}
Note that the energy integrand $E$ (and not only the integrated energy) is invariant under gauge transformations
which are constant in $\lambda$. In particular, we can use any $\lambda$-independent frame to compute
the energy. Similarly, the Willmore energy can be computed with respect to any frame of the (complexified) $\underline{\R^{4,1}}$-bundle.
\end{remark}

\begin{proof}[Proof of Proposition \ref{prowiin}.] By Lemma \ref{con1SN}
 the Willmore integrand is locally given by
\[
-\tfrac{1}{4}\text{tr}(*\beta\wedge\beta) = -\tfrac{i}{2}\mathrm{tr}(\beta^{1,0}\wedge\beta^{0,1})=2i q\bar q e^{-2u} dz\wedge d\bar z.
\]
On the other hand, locally, and with respect to the chosen $\SU(2)$ frame, $\hat\Phi$ and $\hat\Psi$ are given by
\[\hat \Phi=\begin{pmatrix} 0&0\\e^{-u} q&0\end{pmatrix} dz \quad \text{ and } \quad  \hat \Psi=\begin{pmatrix} 0&-e^{-u} \bar q\\ 0&0\end{pmatrix} d\bar z,\]
and we obtain
\[
-2i\, \text{tr}(\hat \Phi\wedge \hat\Psi)=2iq\bar q e^{-2u} dz\wedge d\bar z = -\tfrac{1}{4}\mathrm{tr}(*\beta\wedge\beta),
\]
proving Proposition \ref{prowiin}.
\end{proof}

\begin{proof}[Proof of Theorem \ref{willmoreenergy}:]
Let $s$ be a section satisfying the assumptions in Theorem \ref{willmoreenergy},
with associated equivariant Willmore surface $f\colon \tilde\Sigma\to S^3 = P\mathcal L$.
We start with a lift of $s$
\[
\nabla^\lambda=d+\eta^\lambda=d+\eta_{-1}\lambda^{-1}+\eta_0+\dots,
\]
where $\eta^\lambda$ is a $\lambda$-family of  $\mathfrak{sl}(2,\C)$-valued 1-forms on $\Sigma$. 
There exists a  curve $\gamma\subset \Sigma$ such that on $M = \Sigma\setminus \gamma$ we have a holomorphic $\lambda$-family of gauge transformations $g_+(\lambda)$ which extends to $\lambda = 0$ and which gauges $\nabla^\lambda$ into self-duality form. That is, on $M$ we have $\nabla^\lambda.g_+(\lambda) = \lambda^{-1}\phi +\nabla_0 + \lambda\phi^*$, where $(\nabla_0,\phi)$ solves the self-duality equations and $\phi$ is still nilpotent. Denote by $L$ the kernel bundle of the Higgs field $\phi$  with orthogonal complement $L^\perp$.

Choose a complementary bundle $\tilde L^*$ of the kernel bundle $\tilde L=\ker(\eta_{-1})$ and let $\tilde h(\lambda) = \mathrm{diag}(1,\lambda)$ with respect to the splitting $\tilde L^* \oplus \tilde L$. Consider the dual surface construction
\[
\hat\nabla^\lambda = \nabla^\lambda.\tilde h(\lambda) = d+\hat\eta^\lambda.
\]
The family $g_+(\lambda)$ induces a family of gauge transformations
$\hat g(\lambda)$, which gauges $\hat\nabla^\lambda$ into the $\SU(1,1)=\SL(2,\R)$ self-duality form (see the discussion in Remark \ref{rem:dualsurfacereality}). We claim that  $\hat g$ extends holomorphically to $\lambda=0$ as a gauge transformation. To see this, note that we have
\[
\hat g(\lambda)=(\pi^{\tilde L^*}\oplus\lambda \pi^{\tilde L}) g_+(\lambda)(\pi^{L^\perp}\oplus \tfrac{1}{\lambda}  \pi^L).
\]
Because $g_+$ maps $L$ to $\tilde L$ it follows that $\hat g$ extends holomorphically to $\lambda=0$ as a  gauge transformation.

On the complex rank 4 bundle $\underline{\mathfrak{gl}(2,\C)}\to\Sigma$ we  consider the family of flat connections $\hat D^\lambda$ given by
\[ 
\hat D^\lambda A=d A+ \hat\eta^{(-\lambda)}A-A\hat\eta^\lambda.
\]
Note that the pair of gauge transformations 
$\hat g_+(\lambda),\hat g_+(-\lambda)$ induces a 
gauge transformation $\hat G(\lambda)$ on $\underline{\mathfrak{gl}(2,\C)}$ by
\[
\hat G(\lambda):=A\mapsto  g_+(-\lambda)Ag_+(\lambda)^{-1}.
\]
Moreover,  by Lemma \ref{lem:Dlambdadualsurface}, $\hat D^\lambda.\hat G(\lambda)$ and  $\mathcal D^\lambda$ are gauge equivalent by a $\lambda$-independent gauge transformation. The mean curvature sphere family extends smoothly through the singularity set of the equivariant minimal surface $f$ (or likewise $g_+$). Therefore also $\hat G$ extends smoothly through this singularity set as a positive gauge transformation. The Theorem now follows from Remark \ref{energy-integrand-constant} and Proposition \ref{prowiin}.
\end{proof}

\section{Energy estimates}\label{Sec:enes}

Corollary \ref{positivity} gives us a possibility to distinguish the space of twistor lines, i.e., the space of $\tau$-negative admissible holomorphic sections, from the space of $\tau$-positive admissible sections, by looking at the value range of $\mathcal E$. Note that this criterion is much easier to handle in practice
than determining whether a $\tau$-real section is $\tau$-positive or $\tau$-negative.
We shall be able to use $\mathcal E$ also to distinguish the recently discovered new components of $\tau$-negative sections \cite{HH2} from the component of twistor lines. We emphasize that these $\tau$-negative sections cannot be admissible. 
In view of Simpson's question \cite{Si-Hodge}, such a complex-analytic tool to distinguish those components seems desirable.

The first indication that the function $\mathcal E$ does help can be seen in the case of tori, i.e., for $\Sigma$ of genus 1. In this case, the $\SL(2,\mathbb C)$ Deligne--Hitchin moduli space has a 2-fold covering of the $\mathbb C^*$ Deligne--Hitchin moduli space. 
Note that the $\mathcal E$-function is still well-defined in this situation, even if we do not have any irreducible  $\lambda$-connections at all:
It is well-known that on a torus solutions of the self-duality equations are totally reducible.
Applying Hitchin's spectral curve approach \cite{HitchinHM} to this situation, we easily see that twistor lines correspond to spectral data of spectral genus 0.
Other components of the space of $\tau$-negative holomorphic sections are given by spectral data for spectral curves of positive genus, compare with \cite{BabBob,HH2, Hitchin_1987}. While the spectral genus
 distinguishes the different components of $\tau$-negative sections, 
the following theorem indicates the use of the $\mathcal E$-function in this context.

\begin{theorem}\label{willmoreestimate}
Let $s$ be a holomorphic section of the (singular) Deligne--Hitchin moduli space over a Riemann surface of genus $1$ which is $\tau$-negative and has a nilpotent Higgs field. Assume that the spectral genus is bigger than $1$.
 Then 
\[
\mathcal E(s)\geq \frac{1}{4}.
\]
\end{theorem}
\begin{proof}
Such sections give rise to M\"obius equivariant Willmore surfaces $f\colon\tilde \Sigma\to S^3$ into the conformal 3-sphere, see \cite{BabBob} or also \cite{HH2}.
Because the kernel bundle of the nilpotent Higgs field on a torus has degree 0, we obtain from Theorem \ref{willmoreenergy} that
\[\mathcal E(s)=\frac{1}{4\pi}\mathcal W(f),\]
where $\mathcal W(f)$ is the Willmore energy of a fundamental piece  of $f$.
The theorem follows from an application of the  quaternionic Pl\"ucker estimate, see \cite[Equation (89)]{FLPP}:
That the spectral genus is at least 2 (in fact it must be odd) implies that there are two quaternionic holomorphic linearly independent sections
on an unbranched 4-fold covering of the torus of a quaternionic holomorphic line bundle. The Willmore energy of this quaternionic holomorphic line bundle coincides 
with the Willmore energy of $f$ on a fundamental piece.
\end{proof}

Note that holomorphic sections with nilpotent Higgs field on a torus cannot be totally reducible and therefore are not twistor lines. They therefore lie in a different component of the space of $\tau$-negative sections than the twistor lines. 
The assumption on the spectral genus in Theorem \ref{willmoreestimate}
leaves open the case of spectral genus 1. In that case, as the solutions  are equivariant, one can make the energy $E(s)$
to be arbitrarily close to $0$ by changing the conformal type of the torus $\Sigma$. On the other hand, it does not seem possible to fix the Riemann surface $\Sigma$ of genus 1 and then find, for each $\epsilon>0$, a $\tau$-negative holomorphic  section $s$ 
in the Deligne--Hitchin moduli space with nilpotent Higgs field  such that $\mathcal E(s)<\epsilon$. 

In general, one might try to use the energy to distinguish different components of $\tau$-negative holomorphic  sections of the Deligne--Hitchin moduli space.
%
A first result is given in the following theorem, where we show that the energy is positive for the $\tau$-negative holomorphic sections constructed in \cite{HH2}. 

\begin{theorem}
There exist Riemann surfaces $\hat\Sigma$ of sufficiently large genus $g\geq2$ whose $\SL(2,\C)$ Deligne--Hitchin moduli space admits irreducible
 $\tau$-negative holomorphic sections $s$ with
\[
\mathcal E(s)>0.
\]
In particular, these sections cannot be twistor lines.
\end{theorem}
\begin{proof}
The non-admissible $\tau$-real holomorphic sections have been constructed 
by a deformation of finite gap solutions of the {\em cosh-Gordon} equation of spectral genus 1 on a torus $\Sigma$. 
The initial section on the torus $\Sigma$ yields an equivariant Willmore surface $f$. By Theorem \ref{willmoreenergy}, the Willmore energy of a fundamental piece is the energy of the section, since the degree of the kernel bundle $L$ is necessarily $0$. Because the Hopf differential $q(dz)^2$ does not vanish, the Willmore integrand is positive, which implies that the Willmore energy of $f$ is positive.
  
The $\tau$-negative holomorphic sections $s$ on surfaces of high genus have been constructed as follows (see \cite[Theorem 4.5]{HH2} for details):
 There is a $q$-fold covering Riemann surface $\hat\Sigma\to \Sigma$ of the initial torus, branched over  the four half-lattice points with branch order $q-1.$ 
 On $\Sigma$, there is a holomorphic family of connections with regular singularities at the half-lattice points and local monodromies in the conjugacy class of
 \[
 \begin{pmatrix} e^{2\pi i/q}&0\\0& e^{-2\pi i/q}\end{pmatrix}.
 \]
 The pull-back of this family of flat connections to $\hat\Sigma$ can be desingularized, and yields a lift of a $\tau$-negative holomorphic section $s$ on $\hat\Sigma$. This gives rise to a branched equivariant Willmore surface $\hat f$ which is minimal in $H^3$ away from its intersection with the boundary at infinity \cite[Section 5]{HH2}.
 The counting of branch orders in
  \cite[Theorem 3.3]{HHSch} also holds in the case of (equivariant) minimal surfaces $\hat f$ constructed by the $\tau$-negative holomorphic sections $s$, as it only depends on the local analysis near the singular points, and branch orders are given by the vanishing order of the Higgs field.
 In particular, (for odd $q$), this yields
 that (with the notations of  \cite[Theorem 3.3]{HHSch}) 
 \[
 \frac{\tilde p}{\tilde q}=\frac{2/q+1}{4}=\frac{2+q}{4q},
 \]
where $\tilde p = 2+q$ and $\tilde q = 4q$ are coprime since $q$ is odd. Then
 \[
 g(\hat\Sigma)=2q-1.
 \] 
 Moreover, the total branch order of $f$ is
 \[
 4 (\tilde{q}/2-\tilde p-1)=4(q-3).
 \]
 Hence, as the differential of the surface is a holomorphic section of
 \[
 K_{\hat\Sigma} L^2,
 \]
 where $L$ is the kernel bundle of the Higgs field of $s$ on $\hat\Sigma$, we compute
 \[
 \mathrm{deg}(L)=\frac{1}{2}(2-2g(\hat \Sigma)+4 (q-3))=-4.
 \]
 By Theorem \ref{willmoreenergy} it remains to show that the Willmore energy of $\hat f$ is bigger than $16\pi.$ This can be seen as follows: 
 The family of regular singular connections on the torus $\Sigma$ yields a equivariant Willmore surface $\underline f$ on the
 4-punctured torus by the reconstruction method in \cite[Section 5]{HH2}. Putting $q$ many M\"obius-congruent pieces of $\underline f$ together in the conformal 3-space yields the (equivariant) Willmore surface $\hat f$.  By construction $\underline f$ is close to $f$
 away from two branch cuts between the singular points on the torus $\Sigma$. It follows from \cite[Section 5]{HH2} that for every $\epsilon>0$ there exists $\delta>0$ such that for all $q$ with $\tfrac{1}{q}<\delta$ we have $|W(\underline{f}) - W(f)|<\epsilon$. 
 Take $\epsilon$ small such that $\tfrac{1}{2}\mathcal W(f)>\epsilon.$ 
 As the Willmore energy of $f$ is positive (independent of $q$) we obtain
 \[
 \mathcal W(\hat f) = q\mathcal W(\underline f) >q(\mathcal W(f)-\epsilon)>16\pi
 \] 
for $q$ large enough. 
\end{proof}
\begin{remark}
Alternative proofs of the theorem can be given by making use of the 
special coordinates introduced in \cite{HH}. 
\end{remark}

\section*{Acknowledgments}

The first author is supported by the DFG Emmy-Noether grant on "Building blocks of physical theories from the geometry of quantization and BPS states", number AL 1407/2-1. 
The second author was supported by RTG 1670 "Mathematics inspired by string theory and quantum field theory" funded by the Deutsche Forschungsgemeinschaft (DFG) while much of this work was carried out. The second author would also like to thank Lynn Heller and Franz Pedit for discussions about the Willmore functional, and Jun-ichi Inoguchi for discussions about dual surfaces. 

\bibliographystyle{amsplain}
\bibliography{references}

\end{document}